\documentclass{amsart}
\usepackage{paralist}
\usepackage{url}
\usepackage{mathrsfs}
\usepackage{color}
\usepackage{graphicx}
\usepackage{caption}
\usepackage{subcaption}
\usepackage{multirow}
\usepackage{mathtools}
\usepackage{amssymb}
\usepackage{booktabs,caption}

\newcommand{\RR}{{\mathbb R}}

\newcommand{\re}{\mathbb{R}}

\newcommand{\cpx}{\mathbb{C}}

\newcommand{\N}{\mathbb{N}}

\newcommand{\diag}{\mbox{diag}}

\newcommand{\lmd}{\lambda}

\newcommand{\eps}{\epsilon}

\def\af{\alpha}

\newcommand{\reff}[1]{(\ref{#1})}

\newcommand{\prm}{\prime}
\newcommand{\mc}[1]{\mathcal{#1}}

\newcommand{\st}{\mathit{s.t.}}

\newcommand{\bdes}{\begin{description}}
	\newcommand{\edes}{\end{description}}
\newcommand{\bal}{\begin{align}}
\newcommand{\eal}{\end{align}}
\newcommand{\bnum}{\begin{enumerate}}
	\newcommand{\enum}{\end{enumerate}}
\newcommand{\bit}{\begin{itemize}}
	\newcommand{\eit}{\end{itemize}}
\newcommand{\bea}{\begin{eqnarray}}
\newcommand{\eea}{\end{eqnarray}}
\newcommand{\be}{\begin{equation}}
\newcommand{\ee}{\end{equation}}
\newcommand{\baray}{\begin{array}}
	\newcommand{\earay}{\end{array}}
\newcommand{\bsry}{\begin{subarray}}
	\newcommand{\esry}{\end{subarray}}
\newcommand{\bca}{\begin{cases}}
	\newcommand{\eca}{\end{cases}}
\newcommand{\bcen}{\begin{center}}
	\newcommand{\ecen}{\end{center}}
\newcommand{\bbm}{\begin{bmatrix}}
	\newcommand{\ebm}{\end{bmatrix}}
\newcommand{\bmx}{\begin{matrix}}
	\newcommand{\emx}{\end{matrix}}
\newcommand{\bpm}{\begin{pmatrix}}
	\newcommand{\epm}{\end{pmatrix}}
\newcommand{\btab}{\begin{tabular}}
	\newcommand{\etab}{\end{tabular}}

\newtheorem{theorem}{Theorem}[section]

\newtheorem{prop}[theorem]{Proposition}

\theoremstyle{definition}
\newtheorem{example}[theorem]{Example}
\newtheorem{exm}[theorem]{Example}
\newtheorem{alg}[theorem]{Algorithm}

\newtheorem{assumption}[theorem]{Assumption}
\newtheorem{remark}[theorem]{Remark}

\numberwithin{equation}{section}

\begin{document}

\title[Bilevel Polynomial Optimization]
{A Lagrange Multiplier Expression Method for Bilevel Polynomial Optimization}

\author{Jiawang Nie}
\address{Department of Mathematics,
University of California, 9500 Gilman Drive, La Jolla, CA, USA, 92093.}
\email{njw@math.ucsd.edu}

\author{Li Wang}
\address{Department of Mathematics, University of Texas at Arlington,
The University of Texas at Arlington
701 S. Nedderman Drive
Arlington, TX, USA, 76019.
}
\email{li.wang@uta.edu}

\author{Jane J. Ye}
\address{Department of Mathematics and Statistics, University of Victoria, Victoria, B.C., Canada, V8W 2Y2.}
\email{janeye@uvic.ca}

\author{Suhan Zhong}
	\address{Department of Mathematics,
		University of California, 9500 Gilman Drive, La Jolla, CA, USA, 92093.}
\email{suzhong@ucsd.edu}

	

\begin{abstract}
This paper studies bilevel polynomial optimization.
We propose a method to solve it globally by using
polynomial optimization relaxations.
Each relaxation is obtained from the Karush-Kuhn-Tucker (KKT) conditions
for the lower level optimization and the exchange technique
for semi-infinite programming.
For KKT conditions, Lagrange multipliers are represented
as polynomial or rational functions.
The Moment-SOS relaxations are used to solve the polynomial optimization relaxations.
Under some general assumptions, we prove the convergence of the algorithm
for solving bilevel polynomial optimization problems.
Numerical experiments are presented to show the efficiency of the method.
\end{abstract}

\maketitle

\section{Introduction}

This paper considers the bilevel optimization problem in the form
\be \label{bilevel:pp}
 \left\{ \baray{rl}
F^* := \min\limits_{x\in \re^n,y\in \re^p} & F(x,y) \\
\st  &  h_i(x,y) = 0 \, (i \in \mc{E}_1), \\
     &  h_j(x,y)\geq 0 \, (j \in \mc{I}_1), \\
		&  y\in S(x),
\earay \right.
\ee
where  $S(x)$ is the set of optimizer(s) of the  lower level problem
\be  \nonumber
(P_x) \qquad ~~~~~~~~~~~~~
\left\{ \baray{rl}
\min\limits_{ z\in \re^p}  & f(x,z)  \\
\st &  g_i(x,z) = 0 \, (i \in \mc{E}_2 ), \\
    &  g_j(x,z) \geq 0\, (j\in \mc{I}_2).
\earay\right. \qquad \qquad \qquad \qquad
\ee
In the above, $F(x,y)$ is the upper level objective function and
$h_i(x,y), h_j(x,y)$ are the upper level constraints,
while $f(x,z)$ is the lower level objective function and
$g_i(x,z), g_j(x,z)$ are the lower level constraints.
Here $\mc{E}_1, \mc{I}_1, \mc{E}_2, \mc{I}_2$ are finite index sets
(some or all of them are possibly empty).
For convenience, we denote the feasible set of the lower level problem by
\be \label{def:Z}
Z(x):= \big\{z\in \re^p \mid  g_i(x,z) =0\  (i \in \mc{E}_2 ), \  g_j(x,z) \geq 0\, (j\in \mc{I}_2) \big \}.
\ee
We call (\ref{bilevel:pp}) a \textit{simple} bilevel optimization problem (SBOP)
if $Z(x)\equiv Z$ is independent of $x$, and call (\ref{bilevel:pp})
a \textit{general} bilevel optimization problem (GBOP) if $Z(x)$ depends on $x$.
When all defining functions are polynomials, we call  (\ref{bilevel:pp})
a bilevel polynomial optimization problem.
Throughout the paper, we assume that the solution set $S(x)$
of the lower level problem $(P_x)$ is nonempty for all feasible $x$.

Bilevel optimization has broad applications, e.g.,
the moral hazard model of the principal-agent problem in economics
\cite{mirrlees1999theory}, electricity markets and networks
\cite{bjorndal2005deregulated},
facility location and production problem
\cite{boglarka2016solving},
meta learning and hyper-parameter selection in
machine learning~\cite{Franceschi,K,Liu}.
More applications can be found in the monographs \cite{bard1998practical,dempe2018bilevel,dempe2015bilevel,Shimizu}
and the surveys on bilevel optimization \cite{ColsonMarcotteSavard,Dempebook}
and the references therein.

Bilevel optimization is challenging both theoretically and computationally,
because of the optimality constraint $y \in S(x)$.
The classical (or the first order) approach is to relax this constraint
by the first order optimality condition for the lower level problem.
But solving the resulting single-level problem may not even recover
a stationary point of the original bilevel optimization problem
if the lower level problem is nonconvex; see \cite{mirrlees1999theory}
{and Example \ref{ex:KKTtransfail} for  counterexamples.}
Moreover, even for the case that the lower level optimization is convex,
the resulting single-level problem may not be equivalent to
the original bilevel optimization problem if local optimality is considered and
the lower level multiplier set is not a singleton (see \cite{Dempe2012MP}).

For each $y\in Z(x)$, it is easy to see the following equivalence
(without any assumptions about the lower level optimization, e.g., convexity)
\begin{equation}\label{equivalence}
y \in S(x) \Longleftrightarrow f(x,y)-  v(x)\leq 0
 \Longleftrightarrow  f(x,z)-f(x,y) \geq 0 \ \  \forall z\in Z(x),
\end{equation} where $v(x):=\inf_{z\in Z(x)} f(x,z)$
is the so-called value function for the lower level problem.
We call any reformulation using the first equivalence in (\ref{equivalence})
the value function reformulation, while those using the second equivalence
in (\ref{equivalence}) the semi-infinite programming (SIP) reformulation.
Using  the value function reformulation results in
an intrinsically nonsmooth optimization problem
which never satisfies the usual constraint qualification (\cite{ye1995optimality}).
Despite these difficulties, recent progresses have been made on constraint qualifications and optimality conditions for bilevel optimization problems,
where the lower level optimization is not assumed to be convex; see the work \cite{KuangYe,Dempebook,XuYe2020,Yebook,ye2010new}
and the references therein.

Solving bilevel optimization problems numerically is extremely hard,
since even when all defining functions are linear,
the computational complexity is already NP-hard \cite{BenAyed}.
Most prior methods in the literature are for
mathematical programs with equilibrium constraints (MPECs) {\cite{LuoPangRalph,Outrata}}
and hence can be used only to solve the reformulation of bilevel optimization
by the first order approach.
Recently, some methods for solving bilevel programs
that are not formulated as MPECs are proposed in
\cite{LLNashBilevel,SBLYe2013,mitsos2008global,outrata1990numerical,
xu2013smoothing,xu2014smoothing,XuYeZhang}.

When all defining functions are polynomials, an optimization problem
can be solved globally by the Lasserre type Moment-SOS relaxations
\cite{lasserre2001global,lasserre2009moments,laurent2014optimization,nie2014optimality}.
This motivates the usage of polynomial optimization techniques
for solving bilevel optimization problems globally 
\cite{jeyakumar2016convergent,nie2017bilevel}.

\subsection*{Contributions}

Denote the set containing all upper and lower level constraints:
\be \label{entireset:U}
\mathcal{U}:=\left\{
(x,y) \in \re^n\times \re^p\\
\left|\baray{l}
h_i(x,y) = 0\, (i \in \mc{E}_1), \,
g_i(x,y) = 0\, (i \in \mc{E}_2), \\
h_j(x,y) \ge 0\, (j \in \mc{I}_1), \,
g_j(x,y) \ge 0 \, (j \in \mc{I}_2)
\earay \right.
\right\}.
\ee
Based on the second equivalence in (\ref{equivalence}),
the bilevel optimization (\ref{bilevel:pp}) is  equivalent to
the following single-level optimization problem:
\be \nonumber
(P) \qquad \qquad ~~~~~~~~~~~~~\left\{ \baray{rl}
\min\limits_{x,y}   & F(x,y)  \\
\st &  (x,y) \in \mathcal{U}, \quad  f(x,z)-f(x,y) \geq 0
\quad  \forall \, z \in Z(x).
\earay \right.
\ee
Problem $(P)$ belongs to the class of
the so-called generalized semi-infinite programs
since the set $Z(x)$ is typically infinite and depends on $x$.

To solve $(P)$, one could construct a sequence of
polynomial optimization relaxations $(P_k)$
which have the same objective function as $(P)$ and have the feasible set
$\mathcal{U}_k$ satisfying the nesting containment:
\begin{equation*}
\mathcal{F} \subsetneqq \cdots  \subsetneqq \mathcal{U}_k \subsetneqq  \cdots
\subsetneqq \mathcal{U}_1  \subsetneqq  \mathcal{U}_0 \subsetneqq  \mathcal{U} ,
\end{equation*}
where $\mathcal{F}$ is the feasible set of $(P)$.
Let $(x^{(k)},y^{(k)})$ be a global minimizer of $(P_k)$.
If $y^{(k)}\in S(x^{(k)})$, then $(x^{(k)},y^{(k)})$ is also a global minimizer of $(P)$.
Otherwise, we can add  new constraints to get a tighter relaxation $(P_{k+1})$.
For the sequence $\{(x^{(k)},y^{(k)})\}_{k=1}^\infty$ produced this way,
we expect that its limit or accumulation point is a global minimizer of $(P)$.
This is a kind of exchange technique from SIP \cite{Hettich1993}.
In each relaxation, if we only relax the infinitely many constraints
\[ f(x,z)-f(x,y) \geq 0\quad  \forall z\in Z(x)\]
by finitely many ones of them,
then the convergence would be extremely slow.
{This is because the set $\mathcal{U}$
typically has a much higher dimension than the feasible set $\mathcal{F}$
of $(P)$, since each $(x,y)\in \mathcal{F}$
satisfies the optimality condition $y\in S(x)$ additionally.
For instance, when the lower level optimization $(P_x)$
is unconstrained, every point in $\mathcal{F}$ 
satisfies the first order optimality condition which is a system of $p$ equations,
and hence the set $\mathcal{U}$ is generally
$p$-dimensional higher than $\mathcal{F}$.}
To fasten the convergence significantly,
it was proposed in \cite{nie2017bilevel} to add the Jacobian representation
for the Fritz John conditions of the lower level problem into each relaxation.
However, the usage of Jacobian representation is typically inconvenient,
because it requires to compute minors of Jacobian matrices.
Moreover, the convergence of the method in \cite{nie2017bilevel} 
is only guaranteed for SBOPs. 
In this paper, we address these difficulties and
give an efficient method for solving GBOPs.

The major motivation for our new method is as follows.
For each $y\in S(x)$, we assume the Karush-Kuhn-Tucker (KKT) conditions hold
\begin{equation*}
\left\{\baray{r}
\nabla_z f(x,y)  -\sum_{j \in \mc{E}_2 \cup \mc{I}_2 }
           \lambda_j\nabla_z g_j(x,y) = 0,\\
 \lambda_j \geq 0,\, \lambda_jg_j(x,y)=0 \, (j \in \mc{I}_2 ),
\earay \right.
\end{equation*}
where the $\lmd_j$'s are Lagrange multipliers.
This can be guaranteed if $f$ and all $g_j$ are linear, or by imposing
the LICQ/MFCQ (see Section \ref{Section2.2}).
In the initial relaxation $(P_0)$,
we relax the constraint $y\in S(x)$ to its KKT conditions.
However, if we add the KKT conditions to $ \mathcal{U}$ and minimize
$F(x,y)$ over the original variables $(x,y)$ as well as $\lambda_j$'s,
the number of variables is significantly increased.
This is not practical if there are a large number of constraints.
By using the technique called  {\it Lagrange multiplier expression}
introduced in \cite{njw17a}, we express $\lmd_j$
as a polynomial (or rational) function, say, $\lambda_j(x,y)$.
Then, we choose the initial polynomial optimization relaxation to be
\be  \nonumber
(P_0)~~~~~~~\left\{
\baray{rl}
\min  & F(x,y) \\
\st & h_i(x,y) = 0\, (i \in \mc{E}_1), \,  h_j(x,y) \ge 0 \, (j \in \mc{I}_1), \\
    & g_i(x,y) = 0\, (i \in \mc{E}_2), \,  g_j(x,y) \ge 0 \, (j \in \mc{I}_2), \\
    & \nabla_z f(x,y)  -\sum_{j \in \mc{E}_2 \cup \mc{I}_2 }
           \lambda_j(x,y) \nabla_z g_j(x,y) = 0,\\
    & \lambda_j(x,y) \geq 0,\, \lambda_j(x,y) g_j(x,y) = 0 \, (j \in \mc{I}_2 ) .
\earay
\right.
\ee
Suppose $(x^{(k)},y^{(k)})$ is a global minimizer of $(P_k)$. If $y^{(k)}\in S(x^{(k)})$,
then $(x^{(k)},y^{(k)})$ must be a global minimizer for $(P)$. 
Otherwise, we can find a point $z^{(k)}\in  Z(x^{(k)})$ such that 
\[
f(x^{(k)},z^{(k)})-f(x^{(k)},y^{(k)})<0.
\]
Can we add the following constraint
\be \label{newcut:lower}
f(x,z^{(k)}) - f(x,y) \, \ge \, 0
\ee
to $\mathcal{U}_k$ to get a new relaxation $(P_{k+1})$ whose feasible set is
\be \label{nextiterationfeasibleset}
 \widetilde{\mathcal{U}}_{k+1} \, = \,
 \{(x,y)\in \mathcal{U}_k : f(x,z^{(k)}) - f(x,y) \, \ge \, 0\} \, ?
\ee
Since the point $(x^{(k)},y^{(k)})\not \in \widetilde{\mathcal{U}}_{k+1}$,
the new relaxation $(P_{k+1})$ given as above would be tighter.
For $(P_{k+1})$ to qualify for a relaxation of $(P)$, the feasible set
$\widetilde{\mathcal{U}}_{k+1}$ 
must contain the feasible region $\mathcal{F}$ of $(P)$. For SBOPs,  
i.e., $Z(x)\equiv Z$ is independent of $x$, the inequality \reff{newcut:lower}
holds for all $(x,y)$ satisfying $y\in S(x)$ and hence
$\mathcal{F} \subseteq  \widetilde{\mathcal{U}}_{k+1}$.
However, for GBOPs,  
the condition  $y\in S(x)$ may not necessarily imply
$f(x,z^{(k)})-f(x,y)\geq 0 \ \forall (x,y)\in \mathcal{U}$ unless $z^{(k)}\in Z(x)$.
Hence, the above $\widetilde{\mathcal{U}}_{k+1}$
may not contain the feasible set $\mathcal{F}$.
To fix this issue, we propose to find a polynomial extension of the vector $z^{(k)}$, 
which is a polynomial function $q^{(k)}(x,y)$ 
satisfying $q^{(k)}(x^{(k)},y^{(k)})=z^{(k)}$ and
$q^{(k)}(x,y) \in Z(x)$ for all $(x,y) \in \mc{U}$.
Such a polynomial extension $q^{(k)}(x,y)$ satisfies
\begin{equation}
\label{polyextend}
y\in S(x) \Longrightarrow  f(x,q^{(k)}(x,y)) - f(x,y) \, \ge \, 0.
\end{equation}
Therefore, we replace the feasible set in (\ref{nextiterationfeasibleset}) by
\begin{equation*}
\mathcal{U}_{k+1} :=  \{(x,y)\in \mathcal{U}_k| 
 f(x,q^{(k)}(x,y)) - f(x,y) \, \ge \, 0\}
\end{equation*}
and the next polynomial optimization relaxation is
\be \nonumber
(P_{k+1})~~~~~~~~~~~~~\left\{ \baray{rl}
\min\limits    & F(x,y)  \\
\st &  (x,y) \in \mathcal{U}_k ,  f(x,q^{(k)}(x,y))-f(x,y) \geq 0.
\earay\right.
\ee
Continuing in this way, we either get an optimal solution of $(P)$
for some $k$, or obtain an infinite sequence
$\{(x^{(k)},y^{(k)})\}_{k=1}^\infty$ such that each accumulation point
is a global minimizer of $(P)$.

The paper is organized as follows.
In Section~\ref{sc:prlm}, we review some basic facts
in polynomial optimization as well as constraint qualifications for nonlinear optimization.
In Section~\ref{sec:GBPP}, we propose
a general approach for solving bilevel polynomial optimization.
In Section~\ref{sec:discussion}, we discuss how to get
Lagrange multiplier expressions and the polynomial function $q^{(k)}(x,y)$.
The numerical experiments are reported in Section~\ref{sc:num}.
{Some conclusions and discussions are
given in Section \ref{sec:conclusion}.
}

\section{Preliminaries}\label{sc:prlm}

\subsection*{Notation}
The symbol $\mathbb{N}$  (resp., $\mathbb{R}, \mathbb{C}$)
denotes the set of nonnegative integers (resp., real numbers, complex numbers).
The $\re_+^n$ denotes the nonnegative orthant of $\re^n$.
For a set $S$ and a positive integer $n$, $S^n$ denotes the $n$ Cartesian products of $S$. 
For an integer $n>0$, $[n]:=\{1,\cdots,n\}$. Let $f(x,z)$ 
denote a continuously differentiable function. We use $\nabla f$ 
to denote its whole gradient and $\nabla_z f$ 
to denote its partial gradient with respect to $z$.
For a vector $v:=(v_1,\ldots,v_n)$ in $\mathbb{R}^n$, $\| v \|$
denotes the standard Euclidean norm and  $diag[v]$ denotes an $n$-by-$n$ 
diagonal matrix with the $i$th diagonal entry $v_i$ for all $i\in[n]$.
For $x := (x_1, \ldots, x_n) \in \mathbb{R}^n$ and 
$\af := (\af_1, \ldots, \af_n)\in\mathbb{N}^n$,
denote the monomial
\[
x^\af \, := \, x_1^{\af_1}\cdots x_n^{\af_n}.
\]
For a positive integer $k$, $[x]_k$ 
denotes the vector of all monomials of the highest degree $k$ 
ordered in the graded lexicographic ordering, i.e.,
\[
[x]_k:=(1, x_1, \cdots, x_n, x_1^2, x_1x_2, \cdots, x_n^k).
\]
The symbol $\mathbb{R}[x] := \mathbb{R}[x_1,\cdots,x_n]$
denotes the ring of polynomials in $\mathbb{R}[x]$ with real coefficients.
For a polynomial $p\in \RR[x]$, we use $\deg(p)$ to denote its degree,
while for a tuple of polynomial $p=(p_1,\ldots,p_r)$, $p_i\in\mathbb{R}[x],\ i\in[r]$, we use $\deg(p)$ to denote the highest degree of $p_i$, i.e., $\deg(p) = \max\{\deg(p_1),\ldots,\deg(p_r)\}$. For $k\in \mathbb{N}$,  $\mathbb{R}[x]_k$ denotes the collection of all real polynomials in $x$ with degree at most $k$.
For a given $p\in \RR[x]$,  define the set product
$p\cdot \mathbb{R}[x]:=\{pq|q\in \mathbb{R}[x]\}$.
The symbol $\mathbf{1}_n$ is used to denote an all-one vector in $\mathbb{R}^n$ and $\mathbf{1}$ denotes an all-one vector with the dimension  omitted. Denote by
$I_n$  the $n$-by-$n$ identity matrix.
For an optimization problem,
{\tt argmin} denotes the set of its minimizers.

\subsection{Moment-SOS relaxations}
\label{polynomial optimization}

For a tuple $p=(p_1,\ldots,p_r)$ in $\re[x]$,
$\mbox{Ideal}(p)$ denotes the smallest ideal containing all $p_i$, i.e.,
$
\mbox{Ideal}(p) = p_1 \cdot \re[x] + \cdots + p_r \cdot \re[x].
$
The $k$th {\it truncation} of the ideal $\mbox{Ideal}(p)$,
denoted as $\mbox{Ideal}(p)_k$, is the set
\[
p_1 \cdot \re[x]_{k-\deg(p_1)} + \cdots + {p_r}  \cdot \re[x]_{k-\deg(p_r)}.
\]
The real zero set of $p$ is denoted as
$\mc{V}(p)  := \{x \in \re^n | \,  p(x) = 0 \}.$

A polynomial $\sigma\in\mathbb{R}[x]$ is said to be a sum of squares (SOS) polynomial if
$\sigma = \sigma_1^2+\cdots+\sigma_k^2$ for some $\sigma_1,\ldots,\sigma_k\in\mathbb{R}[x]$.
We use the symbol $\Sigma[x]$ to denote the collection of all SOS polynomials in $x$.
Its $m$th truncation is given by
$\Sigma[x]_m:=\Sigma[x]\cap \mathbb{R}[x]_m$.
We define the \textit{quadratic module} with respect to $q=(q_1,\ldots,q_t)\in (\re[x])^t$ by
\[
\mbox{Qmod}(q) \,:= \, \Sigma[x]+q_1\cdot \Sigma[x]+\cdots+q_t\cdot\Sigma[x].
\]
For $k\in \mathbb{N}$ and $2k\geq deg(q)$, the $k$th truncation of $\mbox{Qmod}(q)$ is
\[
\mbox{Qmod}(q)_{2k}  \,:= \,
\Sigma[x]_{2k}+q_1\cdot\Sigma[x]_{2k-deg(q_1)}+\cdots+q_t\cdot \Sigma[x]_{2k-deg(q_t)}.
\]
For a tuple of polynomials {$q = (q_1,\ldots, q_t)$ in $\mathbb{R}[x]$},
 denote the   basic semi-algebraic set
$\mc{W} (q) := \{x\in\mathbb{R}^n|\ q(x)\geq 0\}.$

Given polynomial tuples $p$ and $q$,
if $f \in \mbox{Ideal}(p)+ \mbox{Qmod}{(q)}$, then it is easy to see that  $f(x)  \geq 0$ for all
$x\in \mc{V}(p) \cap \mc{W}(q)$. 
To ensure $f \in \mbox{Ideal}(p)+ \mbox{Qmod}{(q)}$, we typically need more than $f(x)  \geq 0$ for all
$x\in \mc{V}(p) \cap \mc{W}(q)$. 
The sum $\mbox{Ideal}(p)+ \mbox{Qmod}{(q)}$ is said to be {\it archimedean}
if there exists $b \in \mbox{Ideal}(p)+ \mbox{Qmod}{(q)}$ such that
$\mc{W}(b) = \{ x \in \re^n: b(x) \geq 0\}$ is {a} compact set. It is shown that $f\in \mbox{Ideal}(p)+\mbox{Qmod}(q)$ if $f>0$ on $\mc{V}(p) \cap \mc{W}(q)$ and
$\mbox{Ideal}(p)+\mbox{Qmod}(q)$ is archimedean \cite{putinar1993positive}.
This conclusion is often referenced as \textit{Putinar's Positivstellensatz}. 
When $f$ is only nonnegative (but not strictly positive) on $\mc{V}(p) \cap \mc{W}(q)$, 
we still have $f \in \mbox{Ideal}(p)+\mbox{Qmod}(q)$ 
under some generic conditions (cf. \cite{nie2014optimality}).

We consider the polynomial optimization problem
\begin{equation}\label{ppp}
{f_{min}:}=\min_{x\in\mathbb{R}^n}\ f(x)
\quad \st \quad p(x)=0,\, q(x)\geq 0,
\end{equation}
where $f\in \mathbb{R}[x]$ and $p,q$ are tuples of polynomials.
The feasible set of problem (\ref{ppp}) is $\mc{V}(p) \cap \mc{W}(q)$.
It is obvious that a scalar $\gamma \le f_{min}$ if and only if
$f- \gamma \geq 0$ on $\mc{V}(p) \cap \mc{W}(q)$,
which can be ensured by the membership
$f- \gamma \in \mbox{Ideal}(p)+\mbox{Qmod}(q)$.
The Moment-SOS  hierarchy of semidefinite relaxations for solving problem (\ref{ppp})
is to solve the relaxations
\begin{equation}
\label{dpr}
{f_{k}:}=\max\ \gamma \quad \st \quad
f-\gamma\in \mbox{Ideal}(p)_{2k}+\mbox{Qmod}(q)_{2k},
\end{equation}
for $k=1,2,\ldots$.
The asymptotic convergence $f_k\rightarrow f_{min}$ as $k\rightarrow\infty$
was shown in \cite{lasserre2001global}.
Under the archimedeanness and some classical optimality conditions,
(i.e., linear independence constraint qualification, strict complementarity
and second order sufficiency conditions),
it holds that  $f_k = f_{min}$ for all $k$ big enough,
as shown in \cite{nie2014optimality}.
The optimization problem (\ref{dpr}) can be solved as a semidefinite program
and hence can  be solved by  software packages such as
\texttt{SeDuMi} \cite{sturm1999using}
and {\tt GloptiPoly~3} \cite{henrion2009gloptipoly}.
Moreover, after obtaining solutions for problem (\ref{dpr}),
we can extract an optimizer for (\ref{ppp}) by using
the so-called \textit{flat truncation} condition \cite{nie2013certifying}.

\subsection{Constraint qualifications}\label{Section2.2}
\label{CQKKT}

Consider the  optimization problem
\begin{equation} \label{pro:op}
\left\{ \baray{rl}
\min& b(x) \\
\st & c_i(x) = 0 \, (i\in \mc{E}), \\
& c_j(x) \ge 0 \, (j \in \mc{I}), \\
\earay \right.
\end{equation}
where $b,c_i, c_j: \re^n \rightarrow \mathbb{R}$
are continuously differentiable.
Let $\mc{I}(\bar x):=\{ j\in \mc{I}| c_j(\bar x)=0\}$
be the active index set of inequalities at a feasible point $\bar x$.
The KKT condition is said to hold at $\bar x$
if there exist Lagrange multipliers $\lambda_j$ such that
\[
{\sum}_{j\in \mc{E}\cup \mc{I}}
\lambda_j \nabla c_j(\bar x) = \nabla b(\bar x), \quad
\, \lambda_j \geq 0,\,\, \lambda_j c_j(\bar x)=0 \, (j\in \mc{I}(\bar{x}) ).
\]
A feasible point $\bar x$ is called a KKT point if it satisfies the KKT condition.
A local minimizer must be a KKT point if all functions are linear.
For nonlinear optimization, certain constraint qualifications are required for KKT points.
The linearly independent constraint qualification (LICQ)
is said to hold at $\bar x$ if the gradient set
$\{\nabla c_j(\bar x)\}_{j\in \mc{E}\cup \mc{I}(\bar x)}$
is linearly independent.
The Mangasarian-Fromovitz constraint qualification (MFCQ)
is said to hold at $\bar x$ if the gradients
$\nabla c_j(\bar x) \,( j\in \mc{E} )$
are linearly independent and there exists a vector $d\in \re^n$ satisfying
\[
\nabla c_i(\bar x)^Td =0 \, ( i\in \mc{E} ), \quad
\nabla c_i(\bar x)^T d >0 \, (i\in \mc{I}(\bar x)).
\]
The MFCQ is equivalent to the following statement 
\[
{\sum}_{j\in \mc{E}\cup \mc{I}(\bar x)} \lambda_j \nabla c_j(\bar x) = 0,
\,\, \lambda_j \ge 0 \, (j\in \mc{I}(\bar x) )
\quad \Longrightarrow \quad  \lambda=0.
\]
When the functions $c_i(x) (i\in \mc{E})$ are linear and
$c_j(x) (j \in \mc{I}(\bar x))$ are concave, the Slater's condition
is said to hold if there exists $x_0$ such that
$c_i(x_0)=0 (i\in \mc{E}), c_i(x_0)>0 (i\in \mc{I})$.
The Slater's condition is equivalent to the MFCQ under the convexity assumption.
If the MFCQ holds at a local minimizer $\bar{x}$,
then $\bar{x}$ is a KKT point
and the set of Lagrange multipliers is compact.
If LICQ holds at $\bar{x}$, then the set of Lagrange multipliers is a singleton.
We refer to \cite{Brks} for constraint qualifications
in nonlinear programming.
For SBOPs, it is a generic assumption that each minimizer of
the lower level optimization $(P_x)$ is a KKT point (see \cite{nie2014optimality}).
When $x$ is one-dimensional,
this assumption is also generic (see \cite{JJT86}).
When the dimension of $x$ is bigger than one,
we do not know whether or not this assumption is generic.
However, in our computational experience,
this assumption is often satisfied.

\subsection{Lagrange multiplier representations}
\label{ssc:LagExpr}

Consider the optimization problem \reff{pro:op}
where $b,c_i, c_j$ are real polynomials in $x \in \re^n$.
For convenience, write that
\[
\mc{E} \cup \mc{I} \, = \, \{ 1, \ldots, m\}, \quad
c = (c_1, \ldots, c_m).
\]
The KKT condition for \reff{pro:op} implies that
\be \label{kkt:C(X)lmd=b(x)}
\underbrace{\bbm
\nabla  c_1(x) & \nabla c_2(x) &  \cdots &  \nabla c_m(x) \\
c_1(x) & 0  & \cdots & 0 \\
0  & c_2(x)  & \cdots & 0 \\
\vdots & \vdots & \ddots & \vdots \\
0  &  0  & \cdots & c_m(x)
\ebm}_{C(x) }
\underbrace{\bbm  \lmd_1 \\ \lmd_2 \\ \vdots \\ \lmd_m \ebm}_{\lmd}
=
\underbrace{\bbm  \nabla b(x) \\ 0 \\ \vdots \\ 0 \ebm}_{ \hat{b}(x)} .
\ee
If there exists a polynomial matrix $L(x)$ such that $L(x)C(x) = I_{m}$, then
\[
\lambda(x) \, =  \, L(x) \hat{b}(x).
\]
This gives an explicit expression for Lagrange multipliers as a function of $x$.
When does such a polynomial matrix $L(x)$ exist?
As showed in \cite{njw17a}, it exists if and only if the constraining tuple
$c$ is nonsingular (i.e.,  the matrix $C(x)$ has full column rank
for all complex vectors $x$). The nonsingularity is a generic condition
in the Zariski topology \cite[Prop.~5.7]{njw17a}.

\section{General Bilevel Polynomial Optimization}
\label{sec:GBPP}

In this section, we propose a framework for solving
the bilevel polynomial optimization \reff{bilevel:pp}.
It is based on solving a sequence of polynomial optimization relaxations,
with the usage of KKT conditions and Lagrange multiplier representations.

\subsection{Lagrange multiplier expressions and polynomial extensions}
\label{ssc:twoass}

For convenience, assume the constraining polynomial tuple
in the lower level optimization $(P_x)$ is
$
g \, := \, (g_1(x,z), \ldots, g_{m_2}(x,z)),
$ with  $[m_2]:=\mc{E}_2\cup \mc{I}_2$.
Then the KKT condition for $(P_x)$ implies that
\be \label{lowerkkt:G(x)lmd=f}
\underbrace{\bbm
\nabla_{z}  g_1(x,y) & \nabla_{z} g_2(x,y) &  \cdots &  \nabla_z g_{m_2}(x,y) \\
g_1(x,y) & 0  & \cdots & 0 \\
0  & g_2(x,y)  & \cdots & 0 \\
\vdots & \vdots & \ddots & \vdots \\
0  &  0  & \cdots & g_{m_2}(x,y)
\ebm}_{G(x,y) }
\underbrace{\bbm  \lmd_1 \\ \lmd_2 \\ \vdots \\ \lmd_{m_2} \ebm}_{\lmd}
=
\underbrace{\bbm  \nabla_z f(x,y) \\ 0 \\ \vdots \\ 0 \ebm}_{ \hat{f}(x,y)} ,
\ee
with $\lambda_j \geq 0, j\in \mc{I}_2$.
Because of the dependence on $x$, the above matrix $G(x,y)$
is typically not full column rank for all
complex pairs $(x,y)$. Hence, there may not exist
$L(x,y)$ such that $L(x,y) G(x,y) = I_{m_2}$.
However, rational polynomial expressions
always exist for Lagrange multipliers.
Therefore, we make the following assumption.

\begin{assumption}
\label{asmp:Lag}
Suppose the KKT condition (\ref{lowerkkt:G(x)lmd=f}) holds
for every minimizer of \reff{bilevel:pp}, there exist polynomials
$d_1(x,y),\ldots, d_{m_2}(x,y) {\ge 0}$ on $\mc{U}$
and there are non-identically zero polynomials
$\phi_1(x,y), \ldots, \phi_{m_2}(x,y)$ such that
\be \label{lagexpr:lmdj=pj:d} 
\lmd_j d_j(x,y) \, =\, \phi_j(x,y),
\quad j=1,\ldots, m_2
\ee
for all KKT points $(x,y)$ as in \reff{lowerkkt:G(x)lmd=f}.
\end{assumption}

Suppose there is a polynomial matrix $W(x,y)$ such that
\[
W(x,y) G(x,y) \, = \, \diag[d(x,y)] ,\quad
d(x,y) \,:=\, (d_1(x,y),\ldots, d_{m_2}(x,y)).
\]
Then we can get Lagrange multiplier expressions as in
(\ref{lagexpr:lmdj=pj:d}), since
\begin{equation}
\label{d(xy)lmd=W(xy)hatf}
\diag[d(x,y)]  \lmd = W(x,y) G(x,y) \lmd = W(x,y) \hat{f}(x,y),
\end{equation}
which is the same as
\be \label{formula:d(x,y):phi(x,y)}
d_j(x,y) \lmd_j \, = \, \big( W(x,y) \hat{f}(x,y) \big)_j .
\ee
(The subscript $j$ denotes the $j$th entry.)
The polynomial $\phi_j(x,y)$ in \reff{lagexpr:lmdj=pj:d}
is then $\big( W(x,y) \hat{f}(x,y) \big)_j$.
Let $D(x,y)$ be the least common multiple of $d_1(x,y),\ldots, d_{m_2}(x,y)$
and $D_j(x,y)$ be the quotient polynomial $D(x,y)/d_j(x,y)$.
Under Assumption~\ref{asmp:Lag}, the set of KKT points in
\reff{lowerkkt:G(x)lmd=f} is contained in
\begin{equation}
\label{def:kktset:K}
\mc{K} := \left\{ (x,y)
\left| \begin{array}{c}
D(x,y) \nabla_z f(x,y)-\sum_{j=1}^{m_2}
D_j(x,y) \phi_j(x,y)  \nabla_z g_j(x,y)=0,\\
\phi_j(x,y) \geq 0,\ \phi_j(x,y) g_j(x,y)=0\, (j\in \mc{I}_2)
\end{array}  \right.
\right \}.
\end{equation}
Indeed, the equivalence holds when $d(x,y)$ is positive on $\mathcal{U}$.
If $d_j(\hat{x}, \hat{y})=0$ for some $j$ and $(\hat{x}, \hat{y}) \in \mc{U}$
then $D(\hat{x}, \hat{y})=0$ and hence the equations
in \reff{def:kktset:K} are automatically satisfied.

\begin{assumption}
\label{ass:q(x,y)}
For every pair $(\hat{x}, \hat{y}) \in \mc{U} \cap \mc{K}$
and for every $\hat{z}\in S(\hat{x})$,
there exists a polynomial tuple
$q(x,y) := (q_1(x,y), \ldots, q_p(x,y))$ such that
\be \label{def:q}
q(\hat{x},\hat{y})=\hat{z},\quad
q(x,y)\in Z(x) \quad \forall \, (x, y)\in\mathcal{U}.
\ee
\end{assumption}
We call the function $q(x,y)$ in the above a
{\it polynomial extension} of the point $\hat{z}$ at $(\hat{x}, \hat{y})$.
More details about Lagrange multiplier expressions and polynomial extensions,
required in Assumptions~\ref{asmp:Lag} and \ref{ass:q(x,y)},
will be given in Section~\ref{sec:discussion}.

\subsection{An algorithm for bilevel polynomial optimization}
\label{ssc:alg}

Under Assumptions~\ref{asmp:Lag} and \ref{ass:q(x,y)},
we propose the following algorithm to
solve the bilevel polynomial optimization \reff{bilevel:pp}.
Recall that $Z(x)$ and $\mc{U}$ are the sets as in
\reff{def:Z} and \reff{entireset:U}, respectively.
We refer to Section~\ref{polynomial optimization}
for the Moment-SOS hierarchy.

\begin{alg}
\label{alg:GBPP}
For the given polynomials $F(x,y), h_i(x,y), f(x,z), g_j(x,z)$
in \reff{bilevel:pp}, do the following:

\begin{itemize}

\item [Step~0]
Find rational expressions for Lagrange multipliers
as in \reff{lagexpr:lmdj=pj:d}, for Assumption~\ref{asmp:Lag}.
Let
$
\mc{U}_0  :=  \mc{U} \cap \mc{K},
$
where $\mc{K}$ is the set in \reff{def:kktset:K}. Let $k:=0$.

\item [Step~1] Apply the Moment-SOS hierarchy to
solve the polynomial optimization
\begin{equation}
\label{alg:updated P}
(P_k)\quad\left\{
\baray{rl}
F_k^* := \,\min\limits_{x\in\mathbb{R}^n,y\in\mathbb{R}^p} &  F(x,y) \\
\st \quad &  (x,y) \in \mc{U}_k.
\earay
\right.
\end{equation}
If $(P_k)$ is infeasible, then either \reff{bilevel:pp}
has no optimizers, or none of its optimizers
satisfy the KKT condition \reff{lowerkkt:G(x)lmd=f}
for the lower level optimization.
If it is feasible and has a minimizer,
solve it for a minimizer $(x^{(k)}, y^{(k)})$.

\item [Step~2] Apply the Moment-SOS hierarchy to solve the lower level optimization
\begin{equation}
  \label{alg:Q}
(Q_k)\quad\left\{
\baray{rl}
\upsilon_k^*:=\min\limits_{z \in \re^p}  &
              f(x^{(k)},z) - f(x^{(k)}, y^{(k)}) \\
\st  &  z \in Z(x^{(k)}),\ (x^{(k)},z)\in \mathcal{K}.
\earay
\right.
\end{equation}
for an optimizer $z^{(k)}$.
If the optimal value $\upsilon_k^* = 0$,
then $(x^{(k)}, y^{(k)})$ is an optimizer for \reff{bilevel:pp} and stop.
Otherwise, go to the next step.

\item [Step~3]
Construct $q^{(k)}(x,y)$, a polynomial extension of the vector $z^{(k)}$,
such that
\[
q^{(k)}(x^{(k)},y^{(k)})=z^{(k)}, \quad  q^{(k)}(x,y) \in Z(x)
\quad \forall (x,y)\in \mc{U}.
\]
Update the set $\mc{U}_{k+1}$ as
\begin{equation*}
\mc{U}_{k+1} \, :=  \{(x,y)\in \mc{U}_k | f(x,q^{(k)}(x,y))-f(x,y) \geq 0 \}.
\end{equation*}
Let $k:= k+1$ and go to Step~1.

\end{itemize}

\end{alg}

In Algorithm~\ref{alg:GBPP}, the polynomial optimization problems
$(P_k), (Q_k)$ need to be solved correctly. This can be done by using
the Lasserre type Moment-SOS hierarchy of semidefinite relaxations.
We refer to Section~\ref{polynomial optimization} for the details.
To solve $(P_k)$, the Moment-SOS hierarchy produces
a sequence of convergent lower bounds for
$F_k^*$, say, $\{\rho_l\}_{l=1}^{\infty}$, such that
\[
\rho_1\leq \cdots\leq \rho_l\leq \cdots\leq F_k^*,\quad 
\lim_{l\rightarrow\infty}\rho_l=F_k^*,
\]
where the subscript $l$ is the relaxation order.
For generic polynomial optimization problems, it has finite convergence,
i.e., $\rho_l=F_k^*$ for some $l$. To check the convergence,
we need to extract a feasible point $(\hat{x},\hat{y})$
such that $\rho_l=F(\hat{x},\hat{y})=F_k^*$. It was shown in
\cite{nie2013certifying} that
the flat truncation condition is a sufficient (and almost necessary) criterion
for detecting the convergence. When the flat truncation condition is met,
the Moment-SOS relaxation is tight and one (or more) minimizer $(x^{(k)}, y^{(k)})$
can be extracted for $(P_k)$.
The lower level polynomial optimization $(Q_k)$
can be solved in the same way by the Moment-SOS hiearchy.

It was shown in \cite{nie2014optimality} that
the hierarchy of Moment-SOS relaxations has finite convergence,
under the archimedeanness and some classical optimality
conditions (i.e., the LICQ, strict complementarity and
second order sufficiency condition).
As a special case, this conclusion can also be applied to the
sub-optimization problem $(P_k)$ in Algorithm~\ref{alg:GBPP},
in particular when the lower level optimization has no inequality constraints
(i.e., $\mc{I}_2=\emptyset$), to ensure the finite convergence.
However, when $\mc{I}_2\not =\emptyset$, there is a complementarity constraint,  
so the problem $(P_k)$ is a mathematical program with complementarity constraints.
It is known that the usual constraint qualification such as MFCQ and LICQ
will never hold for such problems (see \cite[Proposition 1.1]{yezhuzhu}
and hence the current theory is not applicable to guarantee the finite convergence.
Therefore, we are not sure whether or not the
Moment-SOS hierarchy has finite convergence for solving $(P_k)$,
when \reff{bilevel:pp} is given by generic polynomials.
We remark that it is possible that the
Moment-SOS hierarchy fails to have finite convergence for
some special cases of $(P_k)$. For instance, this is the case
if $F$ is the Motzkin polynomial and \reff{bilevel:pp}
has a ball constraint and all $f,g_i, g_j$ are zero polynomials
(see \cite{nie2014optimality}).
However, in our computational experience,
the sub-optimization problem $(P_k)$
is almost always solved successfully by the Moment-SOS hierarchy.
In contrast, the sub-optimization problem $(Q_k)$
is easier to be solved by the Moment-SOS hierarchy.
This is because Lagrange multiplier expressions for $(P_x)$
are used to formulate $(Q_k)$.
The Moment-SOS hierarchy has finite convergence
for almost all cases. This is implied by results in \cite{njw17a}.

In addition to the Lasserre type Moment-SOS relaxations,
there exist other types of relaxations for solving polynomial optimization.
For instance, the SOCP relaxations based on SDSOS polynomials \cite{AhmadiSDSOS2019},
the bounded degree SOS relaxations \cite{LasserreBoundedSOS},
or a mixture of them \cite{ChuongSOCP}. In principle, these relaxation methods
can also be used in Algorithm~\ref{alg:GBPP}.
However, we would like to remark that the performance of these relaxations
is much worse than the classical Lasserre type Moment-SOS relaxations.
Such a comparison is done in Example \ref{outrata1994ex31}.
A major reason is that these other types of relaxations
cannot solve the sub-optimization problems $(P_k)$ or $(Q_k)$
accurately enough. Note that Algorithm~\ref{alg:GBPP} requires that
global optimizers of $(P_k)$ and $(Q_k)$ are computed successfully.

\subsection{Convergence analysis}
\label{ssc:convergence}

We study the convergence of Algorithm~\ref{alg:GBPP}.
First, we show that if  the problem $(P_x)$ is convex for each $x$,
then Algorithm~\ref{alg:GBPP} will find a global optimizer of
the bilevel optimization \reff{bilevel:pp} in the initial loop.

\begin{prop}
\label{prop:lower convex}
Suppose that Assumptions~\ref{asmp:Lag} and \ref{ass:q(x,y)}
hold and all $d_j(x,y)>0$ on $\mc{U}$.
For every given $x$, assume that $f(x,z)$ is convex with respect to $z$,
$g_i(x,z)$ is  linear in $z$ for $i \in \mc{E}_2$,
and $g_j(x,z)$ is concave in $z$ for $j \in \mc{I}_2$.
Assume that the Slater's condition holds for $Z(x)$ for all feasible $x$.
Then, the bilevel optimization (\ref{bilevel:pp}) is equivalent to
$(P_0)$ and Algorithm \ref{alg:GBPP}
terminates at the loop $k=0$.
\end{prop}
\begin{proof}
Under the given assumptions, $y \in S(x)$ if and only if
$y$ is a KKT point for problem $(P_x)$,
which is then equivalent to $(x,y) \in \mc{K}$,
since all $d_j(x,y) >0$ on $\mc{U}$.
Then, the feasible set of (\ref{bilevel:pp})
is equivalent to $\mc{U} \cap \mc{K}$.
This implies that (\ref{bilevel:pp}) is equivalent to
$(P_0)$ and Algorithm \ref{alg:GBPP}
terminates at the initial loop $k=0$.
\end{proof}

Second, if Algorithm~\ref{alg:GBPP} terminates at some loop $k$,
we can show that it produces a global optimizer
for the bilevel optimization (\ref{bilevel:pp}).

\begin{prop}
Suppose that Assumptions~\ref{asmp:Lag} and \ref{ass:q(x,y)} hold.
If Algorithm \ref{alg:GBPP} terminates at the loop $k$, then the point
$(x^{(k)}, y^{(k)})$ is a global optimizer of (\ref{bilevel:pp}).
\end{prop}
\begin{proof}
By Assumption \ref{asmp:Lag}, the KKT condition (\ref{lowerkkt:G(x)lmd=f})
holds at each $(x,y) \in \mathcal{U}\cap \{(x,y): y\in S(x)\}=\mathcal{F}$
and hence $\mathcal{F} \subseteq \mathcal{U}_0:=\mathcal{U}\cap \mathcal{K}$.
By the construction of $q^{(k)}(x,y)$ as required for Assumptions~\ref{ass:q(x,y)},
we have shown $\mathcal{F} \subseteq \mathcal{U}_{k}$ for each $k$,
by virtue of (\ref{polyextend}).
Hence we have $F_k^*\leq F^*$ for all $k$,
where $F^*$ denotes the optimal value of (\ref{bilevel:pp}).
 According to the stopping rule, if Algorithm \ref{alg:GBPP}
terminates at the $k$th loop, then $y^{(k)}\in S(x^{(k)})$. 
This means  $(x^{(k)}, y^{(k)})\in \mathcal{F}$. 
Consequently $F_k^*=F(x^{(k)},y^{(k)}) \geq F^*$.
Hence $(x^{(k)}, y^{(k)})$ is a global optimizer of (\ref{bilevel:pp}).
\end{proof}

Last, we study the asymptotic convergence of Algorithm \ref{alg:GBPP}.
To prove the convergence, we need to assume that the value function $v(x)$
is continuous at an accumulation point $x^*$. This is the case
under the so-called {\it restricted inf-compactness} (RIC) condition
(see e.g., \cite[Definition 3.13]{GuoLinYeZhang})
and either $Z(x)$ is independent of $x$ or the MFCQ holds at some
$\bar z\in Z(x^*)$; see \cite[Lemma 3.2]{GauvinDubeau}
for the upper semicontinuity and \cite[page 246]{Clarke} for the lower semicontinuity.
The RIC holds at $x^*$ for $v(x)$ if the value $v(x^*)$ is finite and
there exist a compact set $\Omega$ and a positive number $\epsilon_0$,
such that for all $\| x - x^* \| < \eps_0 $ with $v(x) <v(x^*)+\epsilon_0$,
there exists $z\in S(x) \cap \Omega$.
For instance, $v(x)$ satisfies the RIC at $x^*$
(see \cite[\S6.5.1]{Clarke}) under one of the following conditions.
\begin{itemize}
\item The set $Z(x)$ is uniformly compact around $x^*$ (i.e.,
there is a neighborhood
$N(x^*)$ of $x^*$ such that the closure of $\cup_{x\in N(x^*)} Z(x)$ is compact).

\item The lower level objective
$f(x,z)$ satisfies the growth condition, i.e., there exists a positive constant $\delta>0$  such that the set
\[
\left\{z \Bigg|
\baray{l}
g_i(x^*, z)=\alpha_i (i\in \mc{E}), \,
 g_j(x^*, z)=\alpha_j (j\in \mc{I}), \\
 f(x^*, z)\leq \vartheta, \,
\sum_{ i\in \mc{E} \cup \mc{I} } \alpha_i^2 \le \delta
\earay
\right \}
\]
is bounded for all real values $\vartheta$.

\item The objective $f(x,z)$
is weakly coercive in $z$  with respect to $Z(x)$
for all $x$ sufficiently close to $x^*$, i.e.,
there is a neighborhood $N(x^*)$ of $x^*$ such that
\[
\lim_{z\in Z(x), \|z\|\rightarrow \infty}
 f(x,z)=\infty \qquad {\forall  x\in N(x^*) } .
\] 
\end{itemize}

\noindent
The following is the asymptotic convergence result for Algorithm \ref{alg:GBPP}.

\begin{theorem}
\label{thm:alg:convergence}
For Algorithm~\ref{alg:GBPP}, we assume the following:
\bit

\item[(a)]
All optimization problems $(P_k)$ and $(Q_k)$
have global minimizers.

\item[(b)]
The Algorithm~\ref{alg:GBPP}
does not terminate at any loop,
so it produces the infinite sequence
$\{ (x^{(k)}, y^{(k)}, z^{(k)}) \}_{k=0}^\infty$.

\item[(c)] Suppose $(x^*,y^*,z^*)$ is an accumulation point of
$\{ (x^{(k)}, y^{(k)}, z^{(k)}) \}_{k=0}^\infty$ and
the value function $v(x)$ is continuous at $x^*$.

\item[(d)] {The polynomial functions $q^{(k)}(x,y)$
converge to $q^{(k)}(x^*,y^*)$ uniformly for $k \in \N$ as $(x,y) \to (x^*,y^*)$.}

\eit
{Then, $(x^*,y^*)$
is a global minimizer for the bilevel optimization (\ref{bilevel:pp}).}
\end{theorem}
\begin{proof}
Since $(x^*, y^*)$ is an accumulation point of the sequence
$\{ (x^{(k)}, y^{(k)}) \}_{k=0}^\infty$,
there is a subsequence $\{ k_\ell \}$ such that $k_\ell \to \infty$ and
\[
(x^{k_\ell}, y^{k_\ell}, z^{k_\ell})  \, \to \, (x^*, y^*, z^*).
\]
Since each $z^{(k_\ell)} \in Z(x^{(k_\ell)})$, we can see that $z^* \in Z(x^*)$.
The feasible set of $(P_{k_{\ell}})$ contains that of \reff{bilevel:pp},  so
\[
F(x^*, y^*) \, = \, \lim_{\ell \to \infty} \,
 F(x^{(k_\ell)}, y^{(k_\ell)}) \, \leq \, F^* ,
\] where $F^*$ is the optimal value of the bilevel optimization (\ref{bilevel:pp}).
(The polynomial $F(x,y)$ is a continuous function.)
To prove $F(x^*,y^*) \ge F^*$, we show that $(x^*,y^*)$
is feasible for problem \reff{bilevel:pp}.
Define the functions
\be \label{def:phi}
H(x,y,z) \, := \, f(x,z) - f(x,y), \quad
\phi(x,y) \,:=\, \inf\limits_{z\in Z(x)} \, H(x,y,z).
\ee
Observe that $\phi(x,y)  = v(x) - f(x,y) \le 0$ for all $(x,y) \in \mc{U}$
and $\phi(x^*,y^*) =0$ if and only if $(x^*,y^*)$ is feasible for \reff{bilevel:pp}.
Since $v(x)$ is continuous at $x^*$, we have $\phi(x^*, y^*)\leq 0.$
Next, we show that $\phi(x^*, y^*) \geq 0$.
For an arbitrary $k^{\prm} \in \N$, and for all $k_{\ell} \geq k^{\prm}$,
the point $(x^{(k_\ell)},y^{(k_\ell)})$ is feasible for $(P_{k^\prm})$, so
\[
H(x^{(k_\ell)},y^{(k_\ell)},z)\geq 0\quad \forall
z\in  \mc{V}^{(k^\prm)}_{ k_\ell }
\]
where $\mc{V}^{(k^\prm)}_{ k_\ell }$ is the set defined as
\[
\mc{V}^{(k^\prm)}_{ k_\ell } :=  \left\{
q^{(0)}(x^{(k_\ell)},y^{(k_\ell)}), \, q^{(1)}(x^{(k_\ell)},y^{(k_\ell)}),
\ldots, \, q^{(k^\prm-1)}(x^{(k_\ell)},y^{(k_\ell)}) \right \}.
\]
As $\ell \to \infty$, we can get
\be \label{H(x*y*z)>=0}
H(x^*,y^*,z)\geq 0\quad \forall z\in \mc{V}^{(k^\prm)}_{ \ast },
\ee
where the set $\mc{V}^{(k^\prm)}_{ \ast }$ is
\[
\mc{V}^{(k^\prm)}_{ \ast } :=  \left\{ q^{(0)}(x^*,y^*), \, q^{(1)}(x^*,y^*),
\, \ldots, \, q^{(k^\prm-1)}(x^*,y^*) \right \}.
\]
The inequality \reff{H(x*y*z)>=0} holds for all $k^\prm$, so
\be \label{H(x*y*z)>=0:closure}
H(x^*,y^*,z) \geq 0 \quad \forall z \in
T:= \{ q^{(k)}(x^*,y^*) \}_{k \in \N }.
\ee
It follows that
\[
H(x^*,y^*,q^{(k_\ell)} (x^*,y^*))\geq 0.
\]
In Algorithm~\ref{alg:GBPP}, each point $z^{(k_\ell)} \in Z(x^{(k_\ell)})$ satisfies
\[
\phi(x^{(k_\ell)},y^{(k_\ell)}) \, = \, H(x^{(k_\ell)},y^{(k_\ell)},z^{(k_\ell)}).
\]
Therefore, we have
\begin{equation}\label{3.12}
\baray{rcr}
\phi(x^*, y^*) & = & \phi(x^{(k_\ell)},y^{(k_\ell)})+\phi(x^*,y^*)
-\phi(x^{(k_\ell)},y^{(k_\ell)})\\
 & \geq & \Big( H(x^{(k_\ell)},y^{(k_\ell)},z^{(k_\ell)})
 -H(x^*,y^*,q^{(k_\ell)} (x^*,y^*)) \Big) + \\
 &  & \Big( \phi(x^*, y^*)-\phi(x^{(k_\ell)},y^{(k_\ell)}) \Big).
\earay
\end{equation}
Since $z^{(k_\ell)} = q^{(k_\ell)}(x^{(k_\ell)}, y^{(k_\ell)})$,
by the condition (d), we know that
\[
\lim_{\ell \rightarrow \infty} z^{(k_\ell)} =
\lim_{\ell \rightarrow \infty} q^{(k_\ell)}(x^{(k_\ell)}, y^{(k_\ell)})
=\lim_{\ell \rightarrow \infty} q^{(k_\ell)} (x^*,y^*),
\]
\[
H(x^{(k_\ell)},y^{(k_\ell)},z^{(k_\ell)}) - H(x^*,y^*,q^{(k_\ell)} (x^*,y^*))
\rightarrow 0 \quad \mbox{as}\quad \ell\rightarrow \infty,
\]
by the continuity of the polynomial function $H(x,y,z)$  at $(x^*, y^*, z^*)$.
By the assumption, $v(x)$ is continuous at $x^*$, so
$\phi(x,y) = v(x) - f(x,y)$ is also continuous at $(x^*, y^*)$.
Letting $\ell \to \infty$ in (\ref{3.12}), we get $\phi(x^*, y^*)\geq 0$.
Thus, $(x^*, y^*)$ is feasible for \reff{bilevel:pp} and so $F(x^*,y^*) \geq F^*$.
In the earlier, we already proved $F(x^*,y^*) \leq F^*$,
so $(x^*, y^*)$ is a global optimizer of \reff{bilevel:pp},
i.e., $(x^*, y^*)$  is a global minimizer of
the bilevel optimization \reff{bilevel:pp}.
\end{proof}

\begin{remark}
To ensure that the sequence $\{(x^{(k)},\,y^{(k)},\,z^{(k)})\}$ 
has an accumulation point, one may assume it is bounded.  
A sufficient condition for this is that
the set $\mc{U}$ is bounded or the the upper level objective $F(x,y)$
satisfies the growth condition, i.e., the set
\[
\Big \{(x,y)\in \mathcal{U} \cap \mathcal{K} : \, F(x, y)\leq \vartheta \Big \}
\]
is bounded for all value $\vartheta$. The condition (d)
in Theorem~\ref{thm:alg:convergence} can be either checked directly
on $q^{(k)}(x,y)$ or implied by that the degrees and coefficients of polynomials
$q^{(k)}(x,y)$ are uniformly bounded.
{For instance, if the polynomial sequence of $q^{(k)}(x,y)$
has bounded degrees and bounded coefficients,
then $q^{(k)}(x,y)$ must converge to $q^{(k)}(x^*,y^*)$ uniformly
as $(x,y) \to (x^*,y^*)$. As shown in the subsection~\ref{q:discuss},
when the lower level optimization $(P_x)$ has
the box, simplex or annular type constraints,
the $q^{(k)}(x,y)$ can be constructed explicitly,
and the resulting polynomial sequence of $q^{(k)}(x,y)$
has bounded degrees and bounded coefficients.
Therefore, the convergence of Algorithm~\ref{alg:GBPP}
is guaranteed for these cases of $(P_x)$,
when the conditions (a), (b), (c) hold.}
\end{remark}

\section{Constructions of polynomials}
\label{sec:discussion}

In Algorithm~\ref{alg:GBPP}, we need Lagrange multiplier expressions
as in \reff{lagexpr:lmdj=pj:d}, required for Assumption~\ref{asmp:Lag},
and the polynomial function $q(x,y)$,
required in Assumption~\ref{ass:q(x,y)}.
This section discusses how they can be obtained.

\subsection{Lagrange multiplier expressions}
\label{ssc:Lagmulexpr}

Lagrange multiplier expressions (LME) are discussed in \cite{njw17a}.
For the classical single level polynomial optimization \reff{pro:op},
the existence of a polynomial matrix $L(x)$ satisfying
$L(x)C(x) = I_m$ is equivalent to that the constraining polynomial tuple
$c$ is nonsingular. If the feasible set $Z(x)$ of the lower level optimization
$(P_x)$ does not depend on $x$,
i.e., \reff{bilevel:pp} is a SBOP, 
the matrix $G(x,y)$ does not depend on $x$,
and then there exists a polynomial matrix $W(y)$
satisfying $W(y) G(y) = I_{m_2}$ for generic $g$ \cite{njw17a}.
If $Z(x)$ depends on $x$,
there typically does not exist $W(x,y)$ such that $W(x,y) G(x,y) = I_{m_2}$.
This is because the matrix $G(x,y)$ in \reff{lowerkkt:G(x)lmd=f}
is typically not full column rank for all complex $x \in \cpx^n$, $y \in \cpx^p$.
We generally do not expect polynomial expressions
for Lagrange multipliers of $(P_x)$ for GBOPs.

However, we can always find a matrix polynomial $W(x,y)$ such that
\be \label{WG(x,y)=d(x,y)}
W(x,y) G(x,y) \, = \, \diag[d(x,y)],
\ee
for a denominator polynomial vector
\[
d(x,y)  \,:=\, \big( d_1(x,y), \ldots, d_{m_2}(x,y) \big)
\]
which is nonnegative on $\mc{U}$.
This ensures the Assumption~\ref{asmp:Lag}.
The $W(x,y), d(x,y)$ satisfying \reff{WG(x,y)=d(x,y)} are not unique.
In computation, we prefer that $W(x,y), d(x,y)$ have low degrees
and $d(x,y) > 0$ on $\mc{U}$
(or $d(x,y)$ has as few as possible zeros on $\mc{U}$).
We would like to remark that there always exist such
$W(x,y), d(x,y)$ satisfying \reff{WG(x,y)=d(x,y)}.
Note that $H(x,y) := G(x,y)^T G(x,y)$ is a positive semidefinite matrix polynomial.
If the determinant $\det H(x,y)$ is not identically zero
(this is the general case),
then the adjoint matrix $\mbox{adj}\big( H(x,y) \big)$ satisfies
\[
\mbox{adj}\big( H(x,y) \big)  H(x,y) \, = \,
\det H(x,y)  I_{m_2}.
\]
Then the equation \reff{WG(x,y)=d(x,y)} is satisfied for
\[
W(x,y) := \mbox{adj}\big( H(x,y) \big)  G(x,y)^T,  \quad
d(x,y) = \det H(x,y)\mathbf{1}_{m_2}.
\]
The above choice for $W(x,y), d(x,y)$ may not be very practical
in computation, because they typically have high degrees.
In applications, there often exist more suitable choices for
$W(x,y), d(x,y)$ with much lower degrees.

\begin{example}
\label{multi_rep}
Consider the lower level optimization problem
\[
\left\{
\begin{array}{rl}
\min\limits_{y\in\mathbb{R}^2}& x_1y_1+x_2y_2\\
\st& (2y_1-y_2,x_1-y_1,y_2,x_2-y_2)\geq 0.
\end{array}
\right.
\]
The matrix $G(x,y)$ and $\hat{f}(x,y)$ in (\ref{lowerkkt:G(x)lmd=f}) are:
\[
G(x,y) = \begin{pmatrix}
2 & -1 & 0 & 0\\
-1 & 0 & 1 & -1\\
2y_1-y_2 & 0 & 0 & 0\\
0 & x_1-y_1 & 0 & 0\\
0 & 0 & y_2 & 0\\
0 & 0 & 0 & x_2-y_2
\end{pmatrix},\quad
\hat{f}(x,y) = \begin{pmatrix}
x_1\\x_2\\0\\0\\0\\0
\end{pmatrix}.
\]
The equation \reff{WG(x,y)=d(x,y)} holds for the denominator vector
\[
d(x,y) = \big(2x_1-y_2, \, 2x_1-y_2,\, x_2(2x_1-y_2),\, x_2(2x_1-y_2) \big)
\]
and the matrix $W(x,y)$ which is given as follows
\[
\left(\begin{array}{*{6}c}
x_1-y_1 & 0 & 1 & 1 & 0 & 0\\
y_2-2y_1 & 0 & 2 & 2 & 0 & 0\\
(x_2-y_2)(x_1-y_1) & (x_2-y_2)(2x_1-y_2) & x_2-y_2 & x_2-y_2 & 2x_1-y_2 & 2x_1-y_2\\
y_2(y_1-x_1) & y_2(y_2-2x_1) & -y_2 & -y_2 & 2x_1-y_2 & 2x_1-y_2
\end{array}\right).
\]
Note that $d(x,y) \ge 0$ for all feasible $(x,y)$.
\end{example}

In numerical computation, we often choose
$W(x,y), d(x,y)$ in \reff{WG(x,y)=d(x,y)}
to have low degrees and $d(x,y) > 0$ on $\mc{U}$
(or $d(x,y)$ has as few as possible zeros on $\mc{U}$).
Although we prefer explicit expressions for $W(x,y)$ and $d(x,y)$,
it may be too complicated to do that for some problems.
In the following, we give a numerical method for finding $W(x,y)$ and $d(x,y)$.
Select a point $(\hat{x}, \hat{y}) \in \mc{U}$.
For a priori low degree $\ell$, we consider the
following convex optimization in $W(x,y)$, $d(x,y)$:
\be
\label{Lagrep:SOSprog}
\left\{ \baray{rl}
\max &  \gamma_1 + \cdots + \gamma_{m_2} \\
\st  & W(x,y)G(x,y) = \diag[d(x,y)] ,\\
  & d(\hat{x},\hat{y}) = \mathbf{1}_{m_2}, \,
   \gamma_1 \ge 0, \ldots,  \gamma_{m_2} \ge 0, \\
 & W(x,y) \in \Big( \re[x,y]_{2\ell-\deg(G)} \Big)^{m_2 \times (p+m_2)}, \\
 & d_j(x,y) -\gamma_j \in \mbox{Ideal}(\Phi)_{2\ell}  + \mbox{Qmod}(\Psi)_{2\ell}
 \,\, (j \in [m_2]) .  \\
\earay
\right.
\ee
In the above, the polynomial tuples $\Phi,\Psi$ are
\be \label{PhiPsi}
\Phi := \{h_i\}_{i \in \mc{E}_1 } \cup \{g_i\}_{ i \in \mc{E}_2}, \quad
\Psi := \{h_j\}_{j \in \mc{I}_1 } \cup \{g_j\}_{ j \in \mc{I}_2}.
\ee
The first equality constraint in \reff{Lagrep:SOSprog} is (\ref{WG(x,y)=d(x,y)}),
which gives a set of linear constraints about coefficients of $W(x,y), d(x,y)$.
The last constraint implies that each
$d_j(x,y) \ge \gamma_j \ge 0, \forall (x,y)\in \mc{U}$.
The equality $d(\hat{x},\hat{y}) = \mathbf{1}_{m_2}$
ensures that $d(x,y)$ is not identically zero.
As commented in the earlier of this subsection,
we have shown that \reff{Lagrep:SOSprog}
must have a solution if the degree $\ell$ is large enough,
when $G(x,y)^TG(x,y)$ is not identically singular.
In practice, we always start with a low degree $\ell$.
If \reff{Lagrep:SOSprog} is infeasible,
we then increase the value of $\ell$, until it becomes feasible.

\subsection{The construction of polynomial extensions}
\label{q:discuss}

We can construct a polynomial extension, required in Assumption \ref{ass:q(x,y)},
for many bilevel optimization problems.
If $(P_x)$ has linear equality constraints,
we can get rid of them by eliminating variables.
If $(P_x)$ has nonlinear equality constraints,
generally there is  no polynomial $q(x,y)$ satisfying Assumption~\ref{ass:q(x,y)},
unless the corresponding algebraic set is rational.
So, we consider cases that $(P_x)$
has no equality constraints, i.e., the label set $\mc{E}_2 = \emptyset$.
Moreover, we assume the polynomials $g_j(x,z)$
are linear in $z$, for each $j \in \mc{E}_2$.
Recall the polynomial tuples $\Phi,\Psi$ given as in \reff{PhiPsi}.
For a priori degree $\ell$ and for given $\hat{x}, \hat{y}, \hat{z}$,
we consider the following polynomial system about $q$:
\be \label{q:exist}
\left\{ \baray{l} \\
 q(\hat{x},\hat{y})=\hat{z}, \\
  g_j(x, q) \in {\mbox{Ideal}(\Phi)_{2\ell}}+\mbox{Qmod}(\Psi)_{2\ell} \, (j \in \mc{I}_2) ,  \\
q = (q_1, \ldots, q_p) \in \big( \re[x,y] \big)^{p}.
\earay
\right.
\ee
The second constraint in (\ref{q:exist}) implies that 
$g_j(x, q(x,y))\geq 0, \forall (x,y) \in \mc{U},  j\in  \mc{I}_2$.
Hence $q$ obtained  as above must satisfy Assumption \ref{ass:q(x,y)}.
The above program can be solved by the software
\texttt{Yalmip} \cite{lofberg2004yalmip}.

\begin{exm}
\label{ex:simpleQ}
Consider Example \ref{multi_rep} with
\begin{align*}
 &\hat{x} =(1,0),\ \hat{y}=(1,0),\ \hat{z}=(0,0),\\
 &h(x,y) = (3x_1-x_2,x_2,x_2-x_1+1),\\
 &g(x,y) = (2y_1-y_2,x_1-y_1,y_2,x_2-y_2).
\end{align*}
For $\ell=2$, a satisfactory $q:=(q_1,q_2)$ for \reff{q:exist} is
\[
q_1(x,y) =  x_2/3,\quad  q_2(x,y) = 2x_2/3,
\]
because $g(x,q)=\frac{1}{3}(0,h_1(x,y),2h_2(x,y),h_2(x,y))$ and
\[
h_1(x,y), \, h_2(x,y)  \, \in \,
\mbox{Ideal}(\Phi)_{2\ell}+\mbox{Qmod}(\Psi)_{2\ell}.
\]
\end{exm}

For computational convenience, we prefer
explicit expressions for $q(x,y)$.
In the following, we give explicit expressions
for various cases of bilevel optimization problems.

\subsubsection{Simple bilevel optimization}
\label{ssc:sbop}

If the feasible set $Z(x)$ of the lower level optimization $(P_x)$
is independent of $x$, i.e., $Z(x) \equiv Z$,
then we can just simply choose
\[
{q(x,y) \, := \, z}
\]
in Assumption~\ref{ass:q(x,y)}, for all $z \in Z$ and all $(x,y)\in \mc{U}$.
It is a constant polynomial function.
Therefore, Assumption~\ref{ass:q(x,y)} is always satisfied for all
simple bilevel optimization.

\subsubsection{Box constraints}
\label{ssc:box}

A typical case is that the lower level problem $(P_x)$
has box constraints.
Suppose the feasible set $Z(x)$ of $(P_x)$ is given as
\[
l(x) \,\leq \, z \, \leq \, u(x),
\]
where $l(x) := \big(l_1(x), \ldots, l_p(x) \big)$,
$u(x) := \big(u_1(x), \ldots, u_p(x) \big)$.
For every $(\hat{x},\hat{y})\in \mc{U} \cap \mathcal{K}$ and $\hat{z}\in S(\hat{x})$,
we can choose $q:= (q_1,\ldots,q_p)$ as
\[
q_j(x,y) \,:= \, \mu_j l_j(x)+(1-\mu_j)u_j(x),\quad j=1,\ldots, p,
\]
where each scalar
$\mu_j:=(u_j(\hat{x})-\hat{z}_j)/(u_j(\hat{x})-l_j(\hat{x})) \in [0,1]$.
(For the special case that $u_j(\hat{x})=l_j(\hat{x})$,
we have $\hat{z}_j=u_j(\hat{x})=l_j(\hat{x})$
and simply choose $\mu_j = 0$.) Then, for each $j$,
\[
q_j(\hat{x},\hat{y}) =
\mu_j l_j(\hat{x}) + (1-\mu_j) u_j(\hat{x}) = \hat{z}_j.
\]
Clearly, $q(x,y) \in Z(x)$ for all $(x,y)\in \mc{U}$.
The following is a more general case.

\begin{exm}
\label{prop:row4rank}
Suppose the feasible set $Z(x)$ of $(P_x)$ is given as
\[
 l(x) \,\le \, Az \,\leq \, u(x) ,
\]
where $A:=[a_1,\ldots,a_{m_2}]^T\in\mathbb{R}^{m_2\times p}$
is a full row rank matrix and $m_2 \leq p$.
Let $a_{m_2+1},\ldots, a_p$ be vectors such that the matrix
\[
B \, :=\, [a_1,\ldots,a_{m_2}, a_{m_2+1},\ldots, a_p]^T \in \re^{p \times p}
\]
is invertible. Then the linear coordinate transformation $z = B^{-1}w$
makes the constraints become the box constraints
$l_j(x)\le w_j \leq \, u(x)_j,\, j\in[m_2]$.
Hence we can choose $q=B^{-1}q'$, where $q' := (q_1',\ldots,q_p')$ as
\[
q_j'(x,y) :=\left\{
\begin{array}{ll}
\mu_jl_j(x)+(1-\mu_j)u_j(x),& j = 1, \ldots, m_2,\\
(By)_j, & j = m_2+1, \ldots, p.
\end{array}
\right.
\]
where each scalar
\[
\mu_j:=(u_j(\hat{x})- (B\hat{z})_j)/(u_j(\hat{x})-l_j(\hat{x})) \in [0,1].
\]
For the special case that $u_j(\hat{x})-l_j(\hat{x})=0$,
we just set $\mu_j=0$.
One can similarly verify that $q(x,y) \in Z(x)$ for all $(x,y)\in \mc{U}$.
\end{exm}

\subsubsection{Simplex constraints}
\label{ssc:simplex}

We consider the case that the lower level optimization $(P_x)$
has the simplex type constraints
\[
l(x) \le z, \quad \mathbf{1}^Tz \leq u(x) ,
\]
where $l(x)$ is a $p$-dimensional polynomial function, $\mathbf{1}$
is the vector of all ones, and $u(x)$ is a scalar polynomial function in $x$.
For every $(\hat{x},\hat{y})\in\mc{U}$ and $\hat{z}\in S(\hat{x})$,
we can choose $q:= (q_1,\ldots,q_p)$ as
\[
q_j(x,y):=c_j \cdot \big( u(x)-\mathbf{1}^Tl(x) \big) +l_j(x),\quad j=1,\ldots, p
\]
where each scalar
$c_j:=(\hat{z}_j-l_j(\hat{x}))/(u(\hat{x})-\mathbf{1}^Tl(\hat{x})) \ge 0$.
(For the special case that $u(\hat{x})-\mathbf{1}^Tl(\hat{x})=0$,
we just simply set all $c_j=0$.)
Note that
\[
q_j(\hat{x}, \hat{y}) \,= \,
c_j \big( u(\hat{x}) - \mathbf{1}^Tl(\hat{x})\big) + l_j(\hat{x}) \,= \, \hat{z}_j.
\]
For all $(x,y) \in \mc{U}$, it is clear that $q(x,y)\geq l(x)$.
Moreover, we also have
\[
\mathbf{1}^Tq(x,y)
= \mathbf{1}^Tl(x) (1 - \sum_{j=1}^p c_j )
+ (\sum_{j=1}^p c_j) u(x) \le  u(x),
\]
since $\mathbf{1}^Tl(x) \le u(x)$ and $c_1+\cdots+c_p\leq 1$.
Therefore, $q(x,y)\in Z(x)$ for all $(x,y) \in \mc{U}$.
In the above, $\mathbf{1}$ can be replaced by a nonnegative vector.
The following is the more general case.

\begin{exm} \label{q:gsimplex}
Suppose that the feasible set $Z(x)$ of $(P_x)$ is given as
\[
 a^Tz \leq u(x), \quad
z_j \geq  l_j(x)\, (j=1,\ldots, p),
\]
where $a :=(a_1, \ldots, a_p) \in \re_+^p$,
$u(x)$ and all $l_j(x)$ are polynomials in $x$.
We can choose $q:= (q_1,\ldots,q_p)$ as
\[
q_j(x,y):=c_j \cdot \big( u(x)-a^T l(x) \big)+l_j(x),
\]
where each $c_j:=(\hat{z}_j-l_j(\hat{x}))/(u(\hat{x})-a^Tl(\hat{x}))\geq 0$.
In particular, we set all $c_j=0$ if $u(\hat{x})-a^Tl(\hat{x})=0$.
Note that
\[
q_j(\hat{x}, \hat{y})=l_j(\hat{x})+
c_j \cdot (u(\hat{x})-a^Tl(\hat{x}))
= \hat{z}_j.
\]
For all $(x,y)\in \mathcal{U}$, it is clear that $q(x,y)\geq l(x)$.
In addition, we have
\[
a^Tq(x,y) = a^Tl(x)(1-\sum_{j=1}^p a_jc_j)+(\sum_{j=1}^p a_jc_j)u(x)\leq u(x)
\]
since $a^Tl(x)\leq u(x)$ and $a_1c_1+\cdots a_pc_p\leq 1$. 
Therefore, $q(x,y)\in Z(x)$ for all $(x,y)\in\mathcal{U}$.
\end{exm}

\subsubsection{Annular constraints}
\label{ssc:annular}

Suppose the lower level feasible set is
\[
Z(x) \, = \, \left\{y \in \mathbb{R}^p
\Big|\begin{array}{c}
r(x) \leq \|y-a(x)\|_d \leq R(x)
\end{array}\right\},
\]
where $\|z\|_d:=\sqrt[d]{\sum_{i=1}^p|z_i|^d}$
and $a(x):=[a_1(x),\ldots,a_p(x)]$
is a polynomial vector, and $r(x), R(x)$ are polynomials
such that $0 \le r(x) \le R(x)$ on $\mathcal{U}$. We can choose
\[
q(x,y) :=a(x)+q'(x)s, 
\]
where $q'(x):=\mu_1r(x)+\mu_2R(x)$, $\mu_1,\mu_2$ are scalars such that
\[
\|\hat{z}-a(\hat{x})\|_d=\mu_1r(\hat{x})+\mu_2R(\hat{x}),\quad
\mu_1,\mu_2\geq 0,\quad \mu_1+\mu_2=1,
\]
and $s:=(s_1,\ldots,s_p)$ is the vector such that
\[
s_i:=\frac{ \hat{z}_i-a_i(\hat{x})}{\|  \hat{z}-a(\hat{x})\|_d},
\quad  i=1,\ldots, p.
\]
(For the special case that $\hat{z} = a(\hat{x})$,
we just set all $s_i=p^{-1/d}$.) Then,
\begin{align*}
\hat{z}-q(\hat{x},\hat{y})&=(\hat{z}-a(\hat{x}))-(q(\hat{x},\hat{y})-a(\hat{x}))\\
&=(\hat{z}-a(\hat{x}))-q'(\hat{x})s=0.
\end{align*}
since $q'(\hat{x})= \|\hat{z} - a(\hat{x})\|_d$. Moreover,
\[
\|q(x,y)-a(x)\|_d =\|q'(x)s\|_d=|q'(x)| \cdot \|s\|_d= |q'(x)|.
\]
Because $0 \le r(x) \le R(x)$ on $\mathcal{U}$, we must have
\begin{align*}
r(x) \le   \|q(x,y)-a(x)\|_d  \leq R(x).
\end{align*}
This means that $q(x,y)$ satisfies Assumption~\ref{ass:q(x,y)}.

\section{Numerical experiments}
\label{sc:num}

In this section, we report numerical results of applying Algorithm \ref{alg:GBPP}
to solve bilevel polynomial optimization problems.
The computation is implemented in MATLAB
R2018a, in a Laptop with CPU 8th Generation Intel R CoreTM i5-8250U and RAM 16 GB.
The software {\tt GloptiPoly 3} \cite{henrion2009gloptipoly}
and {\tt SeDuMi} \cite{sturm1999using}
are used to solve the polynomial optimization problems in Algorithm~\ref{alg:GBPP}.
In this section, we use the following notation.

\begin{itemize}

\item
The matrix $G(x,y)$ and vector $\hat{f}(x,y)$ are given as in (\ref{lowerkkt:G(x)lmd=f}).
The polynomials $\phi_j(x,y), d_j(x,y)$ for Lagrange multiplier expressions in Assumption~\ref{asmp:Lag} are given by \reff{d(xy)lmd=W(xy)hatf}, i.e.,
$\phi_j(x,y)$ is the $j$th entry of $W(x,y)\hat{f}(x,y)$,
for a matrix polynomial $W(x,y)$ satisfying \reff{WG(x,y)=d(x,y)}.
In our examples, such $W(x,y)$ is determined by
symbolic Gaussian elimination on the equation \reff{WG(x,y)=d(x,y)}.

\item The notation $(P)$ denotes the bilevel optimization \reff{bilevel:pp}.
Its optimal value and optimizers are denoted by $F^*$ and $(x^*,y^*)$ respectively.

\item The $(P_k)$ denotes the relaxed polynomial optimization in the
$k$th loop of Algorithm~\ref{alg:GBPP}.
Its optimal value and minimizers are denoted as
$F_k^*$ and $(x^{(k)}, y^{(k)})$ respectively.

\item The $(Q_k)$ denotes the lower level optimization problem \reff{alg:Q}
in the $k$th loop of Algorithm~\ref{alg:GBPP}.
Its optimal value and minimizers are denoted as
$\upsilon_k$ and $z^{(k)}$ respectively.

\item We always have $\upsilon_k \le 0$.
Note that $y^{(k)}$ is a minimizer of \reff{alg:Q}
if and only if $\upsilon_k = 0$.
Due to numerical round-off errors, we cannot have $\upsilon_k = 0$
exactly. We view $y^{(k)}$ as a minimizer of \reff{alg:Q}
if $\upsilon_k \ge -\epsilon$, for a tiny scalar $\eps$
(e.g., $10^{-6}$).

\end{itemize}

\begin{exm}
\label{ex:SBPP}
First, we apply Algorithm \ref{alg:GBPP} to solve SBOPs.
The displayed problems are respectively from
\cite[Example 5.2]{SBLYe2013},
\cite[Example 3]{allende2013solving},
 \cite[Example 3.8]{dempe2014solution},
\cite[Example 5.2]{nie2017bilevel} and \cite[Example 2]{ShiAiyoshi}.
All but the first problem are solved successfully in the initial loop $k=0$.
The computational results are shown in Table~\ref{tab:SBPP},
where $\mathtt{\arg\min}$ denotes the set of minimizer(s).
In Table~\ref{tab:SBPP}, we use $v^*$ to denote the value of
$v_k$ in the last loop.
Algorithm \ref{alg:GBPP} computed global optimizers for all of them.
In Table~\ref{tab:SBOPcomparison}, we compare Algorithm~\ref{alg:GBPP}
with some prior methods for solving SBOPs in existing references,
for the quality of computed solutions and the consumed CPU time (in seconds).
For the SBOP in \cite[Example~3]{allende2013solving} and
\cite[Example~3.8]{dempe2014solution},
no CPU time was given in the work,
so we implement their methods with the \texttt{MATLAB} function {\tt fmincon}.
For \cite[Example 3]{allende2013solving}, the method requires to choose starting points.
The performance depends on the choice. We choose $100$ random starting points.
For some of them, the method converges; for the others, it does not.
We report the minimum CPU time for cases that it converges
in Table~\ref{tab:SBOPcomparison}.
In \cite[Example 2]{ShiAiyoshi}, it was mentioned that
$225$ is the true optimal value but the method there cannot compute it accurately.
There is no publicly available code for implementing that method,
so its CPU time is not reported.
\begin{table}[htb]
\centering
\caption{Computational results for some SBOPs.}
\label{tab:SBPP}
\begin{tabular}{l|l}
\specialrule{.2em}{0em}{0.1em}
$\baray{cl}
\min\limits_{x,y\in\mathbb{R}} & x+y\\
\st & (x+1,1-x)\ge 0,\\
& y\in\arg\min\limits_{z\in\mathbb{R}} \frac{1}{2}xz^2-\frac{1}{3}z^3\\
&\qquad\qquad\st\,  (z+1,1-z)\ge 0
\earay
$
& $\begin{array}{l}
F^* = -1.2380\cdot10^{-8}
,\\
v^* = -3.9587\cdot10^{-8},\\
x^* =  -1.0000,\\
y^* = 1.0000.
\end{array}$\\
\hline
\\[-1em]
$\baray{cl}
\min\limits_{x,y\in \re^2}\ & x_1^2-2x_1+x_2^2-2x_2+y_1^2+y_2^2\\
\st & (x_1,x_2,y_1,y_2,2-x_1)\geq 0, \\
& y\in \arg\min\limits_{z\in\mathbb{R}^2}\ z_1^2-2x_1z_1+z_2^2-2x_2z_2\\
&\qquad\qquad \st\ 0.25-(z_1-1)^2\geq 0,\\
&\qquad\ \qquad\quad\  0.25-(z_2-1)^2\geq 0.
\earay$
& $\begin{array}{l}
F^* = -1.0000
,\\
v^* = -1.3113\cdot 10^{-9},\\
x^* = (0.5000,0.5000),\\
y^* = (0.5000,0.5000).
\end{array}$ \\  \hline
\\[-1em]
$\baray{cl}
\min\limits_{x,y\in\mathbb{R}^2}\ & 2x_1+x_2-2y_1+y_2\\
\st & (1+x_1,1-x_1,1+x_2,-0.75-x_2)\geq 0,\\
& y\in \arg\min\limits_{z\in\mathbb{R}^2}\ x^Tz\\
&\qquad\qquad \st \, (2z_1-z_2,2-z_1)\geq 0,\\
&\qquad\ \qquad\quad\ (z_2,2-z_2)\geq 0.
\earay$
&
$\begin{array}{l}
F^* = -5.0000,\\
v^* = -1.4163\cdot10^{-8},\\
x^* = (-1.0000,-1.0000),\\
y^* = (2.0000,2.0000).
\end{array}$\\  \hline
\\[-1em]
$\baray{cl}
\min\limits_{x\in\mathbb{R}^2,y\in\re^3} &
           x_1y_1+x_2y_2+x_1x_2y_1y_2y_3\\
\st & (1-x_1^2,1-x_2^2,x_1^2-y_1y_2)\geq 0,\\
& y\in \arg\min\limits_{z\in\re^3}\, x_1z_1^2+x_2^2z_2z_3-z_1z_3^2\\
&\qquad\qquad \st \quad  (z^Tz-1,2-z^Tz)\geq 0.
\earay$
&
$\begin{array}{l}
F^* = -1.7095,\\
v^* = -1.3995\cdot 10^{-9},\\
x^* = (-1.0000,-1.0000),\\
y^* = (1.1097,0.3143,-0.8184).
\end{array}$ \\  \hline
\\[-1em]
$\baray{cl}
\min\limits_{x,y\in\re^2}\ & (x_1-30)^2+(x_2-20)^2-20y_1+20y_2\\
\st & (x_1+2x_2-30,25-x_1-x_2,15-x_2)\geq 0,\\
& y\in \arg\min\limits_{z\in\re^2}\ (x_1-z_1)^2+(x_2-z_2)^2\\
&\qquad\quad \st \quad  (10-z_1,10-z_2,z_1,z_2)\geq 0.
\earay$
&
$\begin{array}{l}
F^* = 225.0000, \\
v^* = -1.6835\cdot 10^{-9}, \\
x^* = (20.0000,5.0000), \\
y^* = (10.0000,5.0000).
\end{array}$ \\
\specialrule{.2em}{0em}{0.1em}
\end{tabular}
\end{table}
\begin{table}[htb]
\centering
\caption{Comparison with prior methods for some SBOPs}
\label{tab:SBOPcomparison}
\begin{tabular}{l|ll|ll}
\specialrule{.2em}{0em}{0.1em}
\multicolumn{1}{c}{} & \multicolumn{2}{c}{Prior Methods}  &
\multicolumn{2}{c}{Algorithm~\ref{alg:GBPP}}\\
\cmidrule(lr){2-3}\cmidrule(lr){4-5}
 & $F^*$ &  time  & $F^*$ & time\\
 \hline
 \cite[Example 5.2]{SBLYe2013} & $4.7260\cdot10^{-8}$ & 47.48 & $-1.2380\cdot10^{-8}$ & 0.89\\
 \hline
 \cite[Example 3]{allende2013solving} & -1.0000  & 0.05 & -1.0000 & 0.34\\
 \hline
 \cite[Example 3.8]{dempe2014solution} & -5.0000 & 0.06 & -5.0000 & 0.27 \\
 \hline
 \cite[Example 5.2]{nie2017bilevel} & -1.7095 & 13.45 & -1.7095 & 6.43 \\
 \hline
\cite[Example 2]{ShiAiyoshi} & 228.7000 & not available & 225.0000 & 0.27 \\
 \specialrule{.2em}{0em}{0.1em}
\end{tabular}
\end{table}
\end{exm}

\begin{exm}
\label{ex:convex_lower}
Consider the GBOP
\[
\left\{
\baray{cl}
\min\limits_{x,y\in \mathbb{R}^2}  & x_1y_1^3+x_2y_2^3-x_1^2x_2^2\\
\st  & (x_1x_2-1,\, x_1,\, x_2,\, 4-x_1^2-x_2^2-y_1^2-y_2^2)\geq 0,\\
 &   y\in S(x),
\earay
\right.
\]
where $S(x)$ is the optimizer set of
\[
\left\{
\begin{aligned}
\min_{z\in\mathbb{R}^2}\quad & z_1^2+z_2^2-2x_2z_1-x_1x_2z_2\\
\st\quad & (z_1,\, z_2-x_2z_1,\, 2x_1-x_2z_1-z_2)\geq 0.
\end{aligned}
\right.
\]
The polynomial matrix $W(x,y)$ satisfying \reff{WG(x,y)=d(x,y)} is
\[
\begin{pmatrix}
2x_1-2x_2y_1 & 2x_1x_2-2x_2y_2 & 2x_2 & 2x_2 & 2x_2\\
-y_1 & 2x_1-y_2 & 1 & 1 & 1\\
-y_1 & -y_2 & 1 & 1 & 1
\end{pmatrix},
\]
for the denominators
\[
d_1(x,y) = d_2(x,y) = d_3(x,y) = 2x_1>0,\, \forall (x,y)\in\mathcal{U}.
\]
The lower level optimization is convex, for given $x$. According to Proposition \ref{prop:lower convex}, we get the optimizer for this bilevel optimization in the initial loop $k=0$ by Algorithm~\ref{alg:GBPP}.
The computational results are shown in Table~\ref{tab:convex_lower}.
\begin{table}[htb]
\centering
\caption{Computational results for Example~\ref{ex:convex_lower}}
\label{tab:convex_lower}
\begin{tabular}{ll}
\specialrule{.2em}{0em}{0.1em}
$(P_0)$ & $F_0^* = -0.7688$, \\
    & $x^{(0)}=(0.6819, \, 1.7059)$, \,
      $y^{(0)} = (0.3997, \, 0.6819)$, \\
$(Q_0)$ & $\upsilon_0 = -3.3569 \cdot 10^{-7} \rightarrow$ stop. \\
  \hline
  Time & 0.31 seconds, \\
  Output & $F^*=F_0^*,\ x^*=x^{(0)},\ y^*=y^{(0)}$. \\
\specialrule{.2em}{0em}{0.1em}
\end{tabular}
\end{table}
\end{exm}

\begin{exm} \cite[Example~2]{van2003globa}
\label{ex:MuuQuy2003Ex2}
Consider the general bilevel optimization
\[
\left\{
\begin{aligned}
\min_{x\in \mathbb{R}^2,y\in\mathbb{R}^3}\quad
& y_1^2+y_3^2-y_1y_3-4y_2-7x_1+4x_2\\
\st\,\qquad & (x_1,x_2,1-x_1-x_2)\geq 0,\,  y\in S(x),
\end{aligned}
\right.
\]
where $S(x)$ is the optimizer set of
\[
\left\{
\begin{aligned}
\min_{z\in\mathbb{R}^3}\quad & z_1^2+0.5z_2^2+0.5z_3^2+z_1z_2+(1-3x_1)z_1+(1+x_2)z_2\\
\st\quad & (-2z_1-z_2+z_3-x_1+2x_2-2,z_1,z_2,z_3)\geq 0.
\end{aligned}
\right.
\]
The polynomial matrix $W(x,y)$ satisfying \reff{WG(x,y)=d(x,y)} is
\[
\begin{pmatrix}
y_1 & y_2 & y_3 & -1 & -1 & -1 & -1\\
2+x_1+2y_1-2x_2 & 2y_2 & 2y_3 & -2 & -2 & -2 & -2\\
y_1 & 2+x_1+y_2-2x_2 & y_3 & -1 & -1 &
-1 & -1\\
-y_1 & -y_2 & 2+x_1-2x_2-y_3 & 1 & 1 &
1 & 1
\end{pmatrix},
\]
for the denominators $(i=1,2,3,4)$
\begin{align*}
d_i(x,y)  = 2+x_1-2x_2
= 3h_1(x,y)+2h_3(x,y)\geq 0,\ \forall (x,y)\in\mathcal{U}.
\end{align*}
By Algorithm~\ref{alg:GBPP},
we get the optimizer for this bilevel optimization in the initial loop $k=0$.
The computational results are shown in Table~\ref{tab:MuuQuy2003Ex2}.
\begin{table}[htb]
\centering
\caption{Computational results for Example~\ref{ex:MuuQuy2003Ex2}}
\label{tab:MuuQuy2003Ex2}
\begin{tabular}{ll}
\specialrule{.2em}{0em}{0.1em}
$(P_0)$ & $F_0^* = 0.6389$, \\
    & $x^{(0)}=(0.6111, \, 0.3889)$, \,
      $y^{(0)} = (0.0000, \, 0.0000, \, 1.8332)$, \\
$(Q_0)$ & $\upsilon_0 = -6.7295 \cdot 10^{-9} \rightarrow$ stop. \\
  \hline
  Time & 1.09 seconds, \\
  Output & $F^*=F_0^*,\ x^*=x^{(0)},\ y^*=y^{(0)}$. \\
\specialrule{.2em}{0em}{0.1em}
\end{tabular}
\end{table}
\end{exm}

\begin{exm} \cite[Example 5.8]{nie2017bilevel}
\label{NieEx58}
Consider the general bilevel optimization
\[
\left\{
\begin{aligned}
\min_{x,y\in\mathbb{R}^4}\quad & (x_1+x_2+x_3+x_4)(y_1+y_2+y_3+y_4)\\
\st\,\,\quad & (1-x^Tx, x_4-y_3^2, x_1-y_2y_4)\geq 0,\,
 y\in S(x),
\end{aligned}
\right.
\]
where $S(x)$ is the set of optimizer(s) of
\[
\left\{
\begin{aligned}
\min_{z\in\mathbb{R}^4}\quad & x_1z_1+x_2z_2+0.1z_3+0.5z_4-z_3z_4\\
\st\quad & (x_1^2+x_3^2+x_2+x_4-z_1^2-2z_2^2-3z_3^2-4z_4^2, z_2z_3-z_1z_4)\geq 0 .
\end{aligned}
\right.
\]
The matrix polynomial $W(x,y)$ satisfying \reff{WG(x,y)=d(x,y)} is
\[
y_4\cdot\begin{pmatrix}
-y_1y_4 & -y_2y_4 & -y_3y_4 & -y_4^2 & 2y_4 & 2y_4\\
2y_1^2-2(x_1^2+x_3^2+x_2+x_4) & 2y_1y_2 & 2y_1y_3 & 2y_1y_4 & -4y_1 & -4y_1
\end{pmatrix},
\]
for the denominators
\begin{align*}
d_1(x,y)  =  d_2(x,y)  &=  2y_4^2 (x_1^2+x_3^2+x_2+x_4)\\
                       &\geq 2y_4^2(y_1^2+2y_2^2+3y_3^2+4y_4^2)\geq 0,\ \forall (x,y)\in\mathcal{U}.
\end{align*}
By Algorithm~\ref{alg:GBPP}, we get the optimizer for this bilevel optimization
in the initial loop $k=0$.
The computational results are shown in Table~\ref{tab:NieEx58}.
\begin{table}[htb]
\centering
\caption{Computational results for Example~\ref{NieEx58}}
\label{tab:NieEx58}
\begin{tabular}{ll}
\specialrule{.2em}{0em}{0.1em}
$(P_0)$ & $F_0^* = -3.5050$,\\ & $x^{(0)} = (0.5442,0.4682,0.4904,0.4942)$, \\
& $y^{(0)} = (-0.7791,-0.5034,-0.2871,-0.1855)$, \\
$(Q_0)$ & $\upsilon_0 =-1.6143 \cdot 10^{-9} \rightarrow$ stop. \\
  \hline
  Time & 49.08 seconds, \\
  Output & $F^*=F_0^*,\ x^*=x^{(0)},\ y^*=y^{(0)}$. \\
\specialrule{.2em}{0em}{0.1em}
\end{tabular}
\end{table}
\end{exm}

\begin{exm}
\label{ex:lin_simplex}
Consider the general bilevel optimization problem
\[
\left\{
\begin{aligned}
\min_{x,y\in\mathbb{R}^4}\quad & x_1^2y_3^2-2x_3x_4+1.2x_1x_3-x_4^2(y_3+2y_4)\\
\st\,\,\,\quad & (\mathbf{1}^Tx,8-\mathbf{1}^Tx,4x_1x_2-y_1^2-y_2^2)\geq 0,\\
& (x_1-y_1,x_2-y_2,4-x_1-x_2,4-x_3^2-x_4^2)\geq 0,\\
& y\in S(x)
\end{aligned}
\right.
\]
where $S(x)$ is the set of optimizer(s) of
\[
\left\{
\begin{aligned}
\min_{z\in\mathbb{R}^4}\quad & {x_1z_1^2+x_2z_2^2+x_3z_3-x_4z_4}\\
\st\quad & (z_1-z_2-x_2,x_1-z_1+z_2,z_1+z_2+x_1+x_2)\geq 0\\
& (4x_1-2x_2-z_1-z_2,z_3,z_4,3-z_3-z_4)\geq 0
\end{aligned}
\right. .
\]
The matrix polynomial $W(x,y)$ satisfying \reff{WG(x,y)=d(x,y)} is
\[
\left(
\begin{array}{*{11}c}
x_1-y_1+y_2 & y_1-y_2-x_1 & 0 & 0 & 2 & 2 & 0 & 0 & 0 & 0 & 0\\
x_2-y_1+y_2 & y_1-y_2-x_2 & 0 & 0 & 2 & 2 & 0 & 0 & 0 & 0 & 0\\
4x_1-2x_2-y_1-y_2 & 4x_1-2x_2-y_1-y_2 & 0 & 0 & 0 & 0 & 2 & 2 & 0 & 0 & 0\\
y_1+y_2+x_1+x_2 & y_1+y_2+x_1+x_2 & 0 & 0 & 0 & 0 & 2 & 2 & 0 & 0 & 0\\
0 & 0 & y_4 & -y_4 & 0 & 0 & 0 & 0 & 0 & 1 & 0\\
0 & 0 & -y_3 &  y_3 & 0 & 0 & 0 & 0 & 1 & 0 & 0\\
0 & 0 & -y_3 & -y_4 & 0 & 0 & 0 & 0 & 1 & 1 & 1
\end{array}
\right),
\]
for the denominator vector
\begin{multline*}
d(x,y) \,= \, \big(2x_1-2x_2,\, 2x_1-2x_2,\, 10x_1-2x_2, \,
     10x_1-2x_2,\, y_4, \, y_3, \, 3 \big)   \\
\,=\, \big(2(g_1(x,y)+g_2(x,y)) \cdot \mathbf{1}_2,\,
           2(g_3(x,y)+g_4(x,y)) \cdot\mathbf{1}_2,\,
           g_6(x,y),\, g_5(x,y),\,3\big) .
\end{multline*}
The denominators are all nonnegative on $\mc{U}$.
By Algorithm \ref{alg:GBPP},
we get the optimizer of this bilevel optimization in the initial loop $k=0$.
The computational results are shown in Table \ref{tab:lin_simplex}.
\begin{table}[htb]
\centering
\caption{Computational results for Example~\ref{ex:lin_simplex} }
\label{tab:lin_simplex}
\begin{tabular}{ll}
\specialrule{.2em}{0em}{0.1em}
$(P_0)$ & $F_0^* = -24.6491$, \\
  & $x^{(0)} = (0.0000,\, 0.0000,\, 0.3204,\, 1.9742)$,\\
  & $y^{(0)} = (0.0000,\, -0.0000,\, 0.0000,\, 3.0000)$, \\
$(Q_0)$ & $\upsilon_0 = -2.5204 \cdot 10^{-8}\rightarrow$ stop; \\
  \hline
  Time & 2.90 seconds\\
  Output & $F^*=F_0^*,\, x^*=x^{(0)},\, y^*=y^{(0)}$. \\
\specialrule{.2em}{0em}{0.1em}
\end{tabular}
\end{table}
\end{exm}

\begin{exm}
\label{ex:simplex}
Consider the general bilevel optimization
\[
\left\{
\baray{cl}
\min\limits_{x,y\in\mathbb{R}^4}  & y_1x_1^2+y_2x_2^2-y_3x_3-y_4x_4\\
\st & (x_1-1, x_2-1, 4-x_1-x_2)\geq 0,\\
& (x_3-1, 2-x_4, x_3^2-2x_4, 8-x^Tx)\geq 0,\\
& y\in S(x)
\earay
\right.
\]
where $S(x)$ is the set of optimizer(s) of
\[
\left\{
\baray{cl}
\min\limits_{z\in \mathbb{R}^4}  & -z_1z_2+z_3+z_4\\
\st & (z_1,z_2,z_3-x_4,z_4-x_3)\geq 0,\\
& (4x_1x_2-x_1z_1-x_2z_2,3-z_3-z_4)\geq 0.
\earay
\right.
\]
The polynomial matrix $W(x,y)$ satisfying \reff{WG(x,y)=d(x,y)} is
\begin{multline*}
\left(\begin{array}{*{10}c}
x_1(4x_2+y_2)-x_1y_1-x_2y_2 & -x_1y_2 & 0 & 0  \\
-x_1y_1 & 4x_1x_2-x_2y_2 & 0 & 0  \\
0 & 0 & 3-x_3-y_3 & x_3-y_4  \\
0 & 0 & x_4-y_3 & 3-x_4-y_4  \\
-y_1 & -y_2 & 0 & 0  \\
0 & 0 & x_4-y_3 & x_3-y_4
\end{array}\right.
\\
\left.\begin{array}{*{10}c}
 x_1 & x_1 & 0 & 0 & x_1 & 0\\
 x_1 & x_1 & 0 & 0 & x_1 & 0\\
   0 & 0 & 1 & 1 & 0 & 1\\
   0 & 0 & 1 & 1 & 0 & 1\\
   1 & 1 & 0 & 0 & 1 & 0\\
   0 & 0 & 1 & 1 & 0 & 1
\end{array}\right),
\end{multline*}
for the denominator vector $d(x,y)$ as follows
\begin{multline*}
d(x,y) = (4x_1x_2+x_1y_2-x_2y_2, 4x_1x_2+x_1y_2-x_2y_2,\\
  3-x_3-x_4, 3-x_3-x_4, 4x_1x_2+x_1y_2-x_2y_2, 3-x_3-x_4).
\end{multline*}
It is clear that $d(x,y)\geq 0$ for all feasible $(x,y)$.
As in the Subsection~\ref{ssc:simplex},
the polynomial function $q:=(q_1,q_2,q_3,q_4)$
in Assumption~\ref{ass:q(x,y)} can be given as
\begin{equation}
\label{q:simplex}
q_1 = \mu_1x_2,\ q_2 = \mu_2x_1,\ q_3 = x_4+\mu_3(3+x_3+x_4),\ q_4=x_3+\mu_4(3+x_3+x_4),
\end{equation}
where (see Subsections~\ref{ssc:simplex}
for the notation $\hat{x}, \hat{y}, \hat{z}$)
\[
\mu_1 = \frac{\hat{z}_1}{\hat{x}_2},\ \mu_2 = \frac{\hat{z}_2}{\hat{x}_1},\ \mu_3 = \frac{\hat{z}_3-\hat{x}_4}{3+\hat{x}_3+\hat{x}_4},\ \mu_4=\frac{\hat{z}_4-\hat{x}_3}{3+\hat{x}_3+\hat{x}_4},
\]
for given $(\hat{x},\hat{y})\in\mathcal{U}$.
Since $x_1,x_2,x_3\geq 1$ and $x_4\geq -2\sqrt{2}$,
the above $\mu_1,\mu_2,\mu_3,\mu_4$ are well defined.
Applying Algorithm \ref{alg:GBPP},
we get the optimizer for this bilevel optimization in the loop $k=1$.
The computational results are shown in Table \ref{tab:annular}.
\begin{table}[htb]
\centering
\caption{Computational results for Example~\ref{ex:simplex}}
\label{tab:annular}
\begin{tabular}{ll}
\specialrule{.2em}{0em}{0.1em}
$(P_0)$ & $F_0^* = -4.4575$,\\
& $x^{(0)}= (1.1548, \, 1.1546, \, 1.6458, \, 1.3542)$,\\
&  $ y^{(0)} =(0.0000, \, 0.0000, \, 1.3542,\, 1.6458)$,\\
$(Q_0)$ & $\upsilon_0 = -5.3362\rightarrow$ next loop; \\
 & $z^{(0)}  =(2.3093,\, 2.3096,\, 1.3542,\, 1.6458)$,\\
 & $q^{(0)}=(2x_2,\, 2x_1,\, x_4,\, x_3)$ as in (\ref{q:simplex}). \\  \hline
  $(P_1)$ & $F_1^*=-0.4574$, \\
   & $x^{(1)} = (1.0000, \, 1.0000, \, 1.6458, \, 1.3542)$,\\
   & $y^{(1)}  = (2.0000, \, 2.0000,\, 1.3542,\, 1.6458)$, \\
  $(Q_1)$ & $\upsilon_1 = -1.9402 \cdot 10^{-9} \rightarrow$ stop. \\
  \hline
  Time & 102.21 seconds, \\
  Output & $F^*=F_1^*,\ x^*=x^{(1)},\ y^*=y^{(1)}$. \\
\specialrule{.2em}{0em}{0.1em}
\end{tabular}
\end{table}
\end{exm}

\begin{exm}
\label{ring_lin}
Consider the general bilevel optimization problem
\[
\left\{ \baray{cl}
\min\limits_{x,y\in\mathbb{R}^4}  & x_1^2y_4^2-x_2y_3^2+x_3y_1-x_4y_2 \\
\st  & (4-x_1^2-x_2^2,\ -x_1-x_2^2,\ y_1-x_1,\ \mathbf{1}^T x)\geq 0,\\
& (x_3+x_4-3,\, 1+x_3-x_4,\, 3-x_3,\, x_4)\geq 0,\\
& y\in S(x),
\earay \right.
\]
where $S(x)$ is the optimizer(s) set of
\[
\left\{
\begin{aligned}
\min_{z\in\mathbb{R}^4}\quad & (x_1-z_1)^2+(x_2-z_2)^2+z_3-z_4\\
\st\quad & 4x_3^2-x_1^2-x_2^2+2x_1z_1+2x_2z_2-z^Tz\geq 0\\
& (z_3,\ x_3-z_3,\ z_4,\ x_4-z_4)\geq 0
\end{aligned}
\right. .
\]
The matrix polynomial $W(x,y)$ satisfying \reff{WG(x,y)=d(x,y)} is
\begin{multline*}
 \left(
\begin{matrix}
-1 & 0 & 0 & 0\\
-(x_3-y_3)y_3 & 0  & (x_3-y_3)(y_1-x_1) & 0\\
y_3^2 & 0 & -y_3(y_1-x_1) & 0\\
-(x_4-y_4)y_4 & 0 & 0 & (x_4-y_4)(y_1-x_1)\\
y_4^2 & 0 & 0 & -y_4(y_1-x_1)
\end{matrix}\right.\\
\left.
\begin{matrix}
 0 & 0 & 0 & 0 & 0\\
 0 & 0 & y_1-x_1 & 0 & 0\\
 0 & y_1-x_1 & 0 & 0 & 0\\
 0 & 0 & 0 & 0 & y_1-x_1\\
 0 & 0 & 0 & y_1-x_1 & 0
\end{matrix}
\right),
\end{multline*}
for the denominator vector
\[
d(x,y) \, = \, (y_1-x_1) \cdot \big(2, \, x_3-y_3, \, y_3, \, x_4-y_4, \, y_4 \big).
\]
It is clear that $d(x,y)\geq 0$ for all feasible $(x,y)$.
The lower level feasible set $Z(x)$ is a mixture of
separable and annular constraints:
\[
Z(x) = \left\{
z\in\mathbb{R}^4\Bigg|\begin{array}{c}
(z_1-x_1)^2+(z_2-x_2)^2+z_3^2+z_4^2\leq 4x_3^2,\\
0\leq z_3\leq x_3,\ 0\leq z_4\leq x_4
\end{array}
\right\}.
\]
As in Subsections~\ref{ssc:box} and \ref{ssc:annular},
the polynomial function $q:=(q_1,q_2,q_3,q_4)$
in Assumption~\ref{ass:q(x,y)} can be given as
\begin{equation}
\label{q:ring_lin}
q_1 = x_1+\mu_1x_3,\ q_2 = x_2+\mu_2x_3,\ q_3 = \mu_3x_3,\ q_4=\mu_4x_4,
\end{equation}
where {(for a given value $(\hat{x}, \hat{y}, \hat{z})$
of $(x^{(k)}$, $y^{(k)}$, $z^{(k)})$, $q$ satisfies $q(\hat{x}, \hat{y})=\hat{z}$)}
\[
\mu_1 = \frac{\hat{z}_1-\hat{x}_1}{\hat{x}_3},\ \mu_2 = 
\frac{\hat{z}_2-\hat{x}_2}{\hat{x}_3},\ \mu_3 = \frac{\hat{z}_3}{\hat{x}_3},\ \mu_4=\frac{\hat{z}_4}{\hat{x}_4 }.
\]
Since $1\le\hat{x}_3\le 3$ and $0\le \hat{x}_4\le 1+\hat{x}_3$,
we have $\mu_4=0$ for the special case when $\hat{x}_4=0$, then
the above $q$ is well-defined.
This bilevel optimization was solved by Algorithm~\ref{alg:GBPP} in the loop $k=1$.
The computational results are shown in Table \ref{tab:ring_lin}.
\begin{table}[htb]
\centering
\caption{Computational results for Example~\ref{ring_lin}}
\label{tab:ring_lin}
\begin{tabular}{ll}
\specialrule{.2em}{0em}{0.1em}
$(P_0)$ & $F_0^* = -41.7143$,\\
 & $x^{(0)} =(-1.5616,1.2496,3.0000,4.0000)$,\\
 & $y^{(0)} = (-1.5616,6.4458,3.0000,0.0008)$, \\
$(Q_0)$ & $\upsilon_0 = -33.9991$, \\
 & $z^{(0)}  =(-1.5615,1.2496,0.0000,4.0000)$, \\
 & $q^{(0)}=(x_1,x_2,0,x_4)$ as in (\ref{q:ring_lin}).\\ \hline
  $(P_1)$ & $F_1^*=-6.0000$,\\
   & $x^{(1)} = (-2.0000,0.0001,3.0000,0.0001)$,\\
   & $y^{(1)}  = (-2.0000,0.0001,-0.0000,0.0001)$, \\
  $(Q_1)$ & $\upsilon_1 = -2.7612 \cdot 10^{-9} \rightarrow$ stop. \\ \hline
  Time & 3.42 seconds, \\
  Output & $F^*=F_1^*,\ x^*=x^{(1)},\ y^*=y^{(1)}$. \\
\specialrule{.2em}{0em}{0.1em}
\end{tabular}
\end{table}
\end{exm}

\section{Conclusions and discussions}
\label{sec:conclusion}

We propose a new method for solving general bilevel polynomial optimization,
which consists of solving a sequence of polynomial optimization relaxations.
Each relaxation is obtained by using KKT conditions
for the lower level optimization.
For KKT conditions, the Lagrange multipliers are represented
as a polynomial or a rational function.
The Moment-SOS relaxations are used to solve each polynomial relaxation,
which is then refined by the exchange technique from semi-infinite programming.
Under some suitable assumptions, we prove the convergence
for both simple and general bilevel polynomial optimization problems.
Numerical experiments are presented to show the efficiency of the method.
In all of our numerical experiments,
the algorithm converges to optimizers in a few loops.
An interesting future work is to explore the complexity of the algorithm.

We would like to emphasis that when the lower level optimization
$(P_x)$ is replaced by its KKT system, the resulting new optimization
may not be equivalent to the original bilevel optimization \reff{bilevel:pp}
in the sense that optimal solutions \reff{bilevel:pp}
may not be recovered by solving the
the initial polynomial optimization relaxation $(P_0)$.
There exists such an example of exponential functions
as in \cite{mirrlees1999theory}.
In the following, we provide a new example of polynomial functions.
\begin{exm}
\label{ex:KKTtransfail}
Consider the SBOP
\[
\left\{
\begin{array}{cl}
\min\limits_{x \in \re^1, y\in\re^1} & xy-y+\frac{1}{2}y^2\\
\st & 1-x^2\ge 0, 1-y^2\ge 0,\\
&  y\in S(x),
\end{array}
\right.
\]
where $S(x)$ is the optimizer set of
\[
\left\{
\begin{array}{cl}
\min\limits_{ z\in\re^1 } & -xz^2+\frac{1}{2}xz^4\\
\st & 1-z^2\ge 0.
\end{array}
\right.
\]
The KKT condition $\nabla_z f(x,z)-\lmd \nabla_z g(z)=0$
for the lower level optimization is
\[
-2xz+2xz^3+2 \lmd z=0.
\]
Therefore, the initial polynomial optimization relaxation $(P_0)$
is equivalent to
\[
\left\{
\begin{array}{cl}
\min\limits_{x,y,\lambda\in\mathbb{R}} & xy-y+\frac{1}{2}y^2\\
\st & 1-x^2\ge 0,\, 1-y^2\ge 0,\\
& -xy+xy^3+\lambda y=0,\, \lambda\ge 0,\,\lambda(1-y^2)=0.
\end{array}
\right.
\]
By a brute force computation,
the above optimization has the optimal value and minimizer respectively
\[
F_{c}^*=-1.5,\quad (x_{c}^*,y_{c}^*)=(-1,1).
\]
However, $(x_c^*,y_c^*)$ is not even feasible for the original bilevel optimization,
since $y_c^* \not\in  S(x_c^*)=\{0\}$.
We can apply Algorithm \ref{alg:GBPP} to solve this SBOP,
with the LME $\lambda(x,y)=x-xy^2$. It terminated in the loop $k=1$.
The computational results are shown in Table~\ref{tab:KKTtransfail}.
\begin{table}[htb]
\centering
\caption{Computational results for Example \ref{ex:KKTtransfail}}
\label{tab:KKTtransfail}
\begin{tabular}{ll}
\specialrule{.2em}{0em}{0.1em}
$(P_0)$ & $F_0^* = -1.5000$,\\
& $(x^{(0)}, y^{(0)}) =(-1.0000, \, 1.0000)$,\\
$(Q_0)$ & $\upsilon_0 = -0.5000\rightarrow$ next loop; \\
 & $z^{(0)}  =1.1385\cdot10^{-17} $,\ $q^{(0)}=z^{(0)}$ as in Section \ref{ssc:sbop}. \\  \hline
  $(P_1)$ & $F_1^*=-0.5000$, \\
   & $(x^{(1)}, y^{(1)})  = (2.4099\cdot 10^{-9}, \, 1.0000)$, \\
  $(Q_1)$ & $\upsilon_1 = -3.5197\cdot10^{-12} \rightarrow$ stop. \\
  \hline
  Time & 0.75 second, \\
  Output & $F^*=F_1^*,\ x^*=x^{(1)},\ y^*=y^{(1)}$. \\
\specialrule{.2em}{0em}{0.1em}
\end{tabular}
\end{table}
\end{exm}

One may consider to solve the sub-optimization problems 
$(P_k)$ and $(Q_k)$ in Algorithm~\ref{alg:GBPP} by methods other than the
classical Lasserre type Moment-SOS relaxations, e.g.,
the bounded degree SOS relaxations (BDSOS) \cite{LasserreBoundedSOS}, and
the bounded degree SOCP relaxations (BDSOCP) \cite{ChuongSOCP} 
that is a mixture of both BDSOS and scaled diagonal SOS (SDSOS) polynomials \cite{AhmadiSDSOS2019}.
Requested by referees, we give a computational comparison
in the following example. We remark that Algorithm~\ref{alg:GBPP}
fails when these two new relaxations are used to solve $(P_k)$.

\begin{exm}\cite[Example 3.1]{outrata1994}
\label{outrata1994ex31}
{
Consider the GBOP
\begin{equation*}
\left\{
\baray{cl}
\min\limits_{x\in\re^1, y\in\re^2} & 0.5(y_1-3)^2+0.5(y_2-4)^2\\
\st & (x,10-x)\geq 0,\, y\in S(x),
\earay
\right.
\end{equation*}
where $S(x)$ is the optimizer set of
\[
\left\{
\baray{cl}
\min\limits_{z\in\mathbb{R}^2} &
0.5(1+0.2x)z_1^2+0.5(1+0.1x)z_2^2-(3+1.333x)z_1-xz_2\\
\st & (0.333z_1-z_2-0.1x+1,\, 9+0.1x-z_1^2-z_2^2,\, z_1,\, z_2)\geq 0.
\earay
\right.
\]
The polynomial matrix $W(x,y)$ satisfying \reff{WG(x,y)=d(x,y)} is
\[
\begin{pmatrix}
6.006y_1y_2^2 & -6.006y_1^2y_2 & 0 & 0 & -6.006y_2^2 & 6.006y_1^2\\
-3.003y_1y_2 & -y_1y_2 & 0 & 0 & 3.003y_2 & y_1\\
0 & 0 & 0 & 0 & 6.006y_1y_2+2y_2^2 & 0\\
0 & 0 & 0 & 0 & 0 & 6.006y_1^2+2y_1y_2
\end{pmatrix},
\]
with the denominators
\[
d_1(x,y) = d_2(x,y) = 6.006y_1^2y_2+2y_1y_2^2\ge 0,\,\forall (x,y)\in \mathcal{U}.
\]
When the classical Lasserre type Moment-SOS relaxations are used
to solve $(P_k)$ and $(Q_k)$ in Algorithm~\ref{alg:GBPP},
we get the correct solution successfully in the initial loop $k=0$.
The computational results are shown in Table~\ref{tab:outrata},
which are the same as in \cite{outrata1994}.
\begin{table}[htb]
\centering
\caption{Computational results for Example \ref{outrata1994ex31}}
\label{tab:outrata}
\begin{tabular}{llllll}
\specialrule{.2em}{0em}{0.1em}
$(P_0)$ & $F_0^* = 3.2077
$,\\ & $x^{(0)}=4.0604,\quad y^{(0)}=(2.6822,1.4871)$,\\
$(Q_0)$ & $v_0 =-3.7906\cdot10^{-6}\rightarrow$ stop \\
  \hline
  Time & 1.031 seconds\\
  Output & $F^*=F_0^*,\ x^*=x_0^*,\ y^*=y_0^*$\\
\specialrule{.2em}{0em}{0.1em}
\end{tabular}
\end{table}}
Now we apply BDSOCP \cite{ChuongSOCP} and BDSOS \cite{LasserreBoundedSOS}
to solve $(P_0)$.
We implement these two new relaxations methods in SPOT~\cite{MegretskiSPOT}
and solve the resulting SOCP and SDP by {\tt MOSEK}.
Both BDSOCP and BDSOS require to use a parameter $M$ at the beginning.
A scale factor $1/M$ will be multiplied to each inequality constraint,
to ensure that the constraining function value is always between 0 and 1.
We remark that estimating such $M$ exactly is quite difficult,
which is equivalent to solving another polynomial optimization problem \cite{ChuongSOCP}. Here, we tune the parameter $M$.
Let $d$ be the relaxation order for both BDSOCP and BDSOS.
Denote by $\hat{F}^*_{d,M}$ and $\tilde{F}^*_{d,M}$
the objective values that are computed by BDSOCP and BDSOS respectively.
The computational results are shown in Table~\ref{tab:outrataCompare}.
The time there is measured in seconds.
None of these two methods solved the initial optimization $(P_0)$ well,
so Algorithm~\ref{alg:GBPP} fails to continue.
\begin{table}[htb]
\centering
\caption{Computational results for Example \ref{outrata1994ex31} with BDSOCP and BDSOS on sub-optimization problem $(P_0)$}
\label{tab:outrataCompare}
\begin{tabular}{l|l|ll|ll}
\specialrule{.2em}{0em}{0.1em}
\multicolumn{2}{c}{} & \multicolumn{2}{c}{BDSOCP}  & \multicolumn{2}{c}{BDSOS}\\
\cmidrule(lr){3-4}\cmidrule(lr){5-6}
$M$ & $d$ & $\hat{F}^*_{d,M}$ &  time  & $\tilde{F}^*_{d,M}$ & time\\
 \hline
 \multirow{3}{*}{10}  & 1 & 2.6011 & 0.40 & 2.6011 & 0.27\\
  & 2 & 3.5516 & 1.03 & 5.6719 & 0.92\\
  & 3 & 5.8476 & 51.17 & 8.4647 & 51.01\\
 \hline
 \multirow{3}{*}{100} & 1 & 2.6011 & 0.16 & 2.6011 & 0.09\\
  & 2 & 3.2121 & 0.86 & 3.5696 & 0.63\\
  & 3 & 3.7170 & 50.45 & 5.8199 & 48.30\\
 \hline
 \multirow{3}{*}{1000} & 1 & 2.6012 & 0.16 & 2.6011 & 0.09\\
  & 2 & 2.7375 & 0.72 & 3.7937 & 0.95\\
  & 3 & 3.7297 & 41.52 & 7.9497 & 39.33\\
 \specialrule{.2em}{0em}{0.1em}
\end{tabular}
\end{table}
This GBOP was not solved accurately by either BDSOCP or BDSOS relaxations.
\end{exm}

In this paper, we assumed the KKT conditions are satisfied at global optimizers
of the lower level optimization $(P_x)$.
When the KKT conditions fail to hold for $(P_x)$,
we do not know how to apply our proposed method.
For such a case, we may consider to use Fritz John conditions
and Jacobian representations as in the work \cite{nie2017bilevel}.
The KKT approach has advantages, as well as potential drawbacks,
for solving bilevel optimization.
We refer to the work \cite{BouzStill13} for this issue.
It is important future work to solve BOPs
when the KKT conditions fail for $(P_x)$.

\end{document}